\documentclass[reqno,12pt]{amsart}

\usepackage[utf8]{inputenc}
\usepackage{times}
\usepackage[UKenglish]{babel}
\usepackage[T1]{fontenc}
\usepackage{stmaryrd,verbatim,calligra,mathrsfs}
\usepackage[normalem]{ulem}
\usepackage{pdflscape,xcolor}

\usepackage[a4paper,top=2cm,bottom=2cm,left=2cm,right=2cm,marginparwidth=1.75cm]{geometry}

\usepackage{amsmath,amssymb,amsfonts,amsthm,upgreek}
\usepackage{mathrsfs, mathtools}
\usepackage{graphicx}
\usepackage{booktabs}
\usepackage{siunitx}
\usepackage{tikz-cd}
\usepackage[colorinlistoftodos,prependcaption, textsize=small, textwidth=2.5cm]{todonotes}
\usepackage[all]{xy}
\usepackage{cancel}
\usepackage{appendix}
\usepackage{enumerate}
\usepackage{multicol}
\usetikzlibrary{calc}

\usepackage[colorlinks=true, allcolors=blue,backref=page]{hyperref}

\newcommand{\cH}{\mathcal{H}}

\newcommand{\cJ}{\mathcal{J}}

\newcommand{\cM}{\mathcal{M}}

\newcommand{\cV}{\mathcal{V}}
\newcommand{\cW}{\mathcal{W}}

\newcommand{\fg}{{\mathfrak g}}

\newcommand{\fh}{{\mathfrak h}}

\newcommand{\fk}{{\mathfrak k}}

\newcommand{\fm}{{\mathfrak m}}

\newcommand{\fp}{{\mathfrak p}}
\newcommand{\fq}{{\mathfrak q}}

\newcommand{\ft}{{\mathfrak t}}
\newcommand{\fu}{{\mathfrak u}}

\newcommand{\fX}{{\mathfrak X}}

\newcommand{\Z}{\mathbb{Z}}

\newcommand{\R}{\mathbb{R}}
\newcommand{\C}{\mathbb{C}}

\newcommand{\bbS}{\mathbb{S}}

\newcommand{\su}{\mathfrak{su}}

\newcommand{\gl}{\mathfrak{gl}}
\newcommand{\fsl}{\mathfrak{sl}}

\renewcommand{\O}{{\rm O}}

\newcommand{\SU}{{\rm SU}}
\newcommand{\GL}{\mathrm{GL}}

\newcommand{\U}{{\rm U}}

\newcommand{\Cas}{\mathrm{Cas}}

\newcommand{\roughtil}{(\nabla^\phi)^\ast\nabla^\phi}

\newcommand{\bangle}[1]{\left\langle #1 \right\rangle}

\newcommand{\Ad}{\mathrm{Ad}}
\newcommand{\Aut}{\mathrm{Aut}}

\renewcommand{\det}{\mathop\mathrm{det}\nolimits}

\renewcommand{\epsilon}{\varepsilon}

\newcommand{\Hom}{{\mathrm{Hom}}}

\newcommand{\ad}{\mathrm{ad}}

\newcommand{\diag}{\mathrm{diag}}

\newcommand{\id}{\mathrm{Id}}

\newcommand{\Ind}{\mathrm{Ind}}

\newcommand{\tr}{\mathop{\mathrm{tr}}\nolimits}

\renewcommand{\geq}{\geqslant}
\renewcommand{\leq}{\leqslant}
\newcommand{\ddto}{\left.\frac{d}{dt}\right\rvert_{t=0}}

\newcommand{\qandq}{\quad\text{and}\quad}

\newcommand{\qforq}{\quad\text{for}\quad}

\def\<{\mathopen{}\left<}
\def\>{\right>\mathclose{}}
\def\({\mathopen{}\left(}
\def\){\right)\mathclose{}}

\usepackage{multicol, color}

\definecolor{gold}{rgb}{0.85,.66,0}
\definecolor{cherry}{rgb}{0.9,.1,.2}
\definecolor{burgundy}{rgb}{0.8,.2,.2}
\definecolor{orangered}{rgb}{0.85,.3,0}
\definecolor{orange}{rgb}{0.85,.4,0}
\definecolor{olive}{rgb}{.45,.4,0}
\definecolor{lime}{rgb}{.6,.9,0}
\definecolor{green}{rgb}{.2,.7,0}
\definecolor{grey}{rgb}{.4,.4,.2}
\definecolor{brown}{rgb}{.4,.3,.1}

\newtheorem{theorem}{Theorem}[section]
\newtheorem*{theorem*}{Theorem}
\newtheorem{proposition}{Proposition}[section]
\newtheorem{corollary}[proposition]{Corollary}
\newtheorem{lemma}[proposition]{Lemma}

\theoremstyle{remark}
\newtheorem{remark}[equation]{Remark}

\theoremstyle{definition}
\newtheorem{definition}[proposition]{Definition}

\numberwithin{equation}{section}

\title{Harmonic maps from $\mathbb{S}^{2n+1}$ to $\C P^n$ with low Morse index}
\author{L. Lara}
\date{\today}

\begin{document}

\begin{abstract}
We study the Jacobi spectrum and low-index rigidity of complex Hopf
fibrations.  We first give a representation-theoretic description of the
Jacobi operator of a homogeneous Riemannian fibration and apply it to
\[
    \pi:(\bbS^{2n+1},g_\epsilon)\longrightarrow(\C P^n,h),
\]
where $g_\epsilon$ is a Berger metric and $h$ is the Fubini--Study metric.
For $n\geq3$, we determine all negative and zero eigenvalues and obtain
explicit formulas for the Morse index and nullity for every
$\epsilon>0$.  In particular, for the unit round metric,
$\mathrm{Ind}(\pi)=2n+2$ and
$\mathrm{Null}(\pi)=n(3n+5)$.  We then extend Rivière's rigidity argument
from $\bbS^3\to\bbS^2$ to higher dimensions.  If
$\phi:\bbS^{2n+1}\to\C P^n$ is harmonic and
$\mathrm{Ind}(\phi)=2n+2$, then
$d\phi_x\circ d\phi_x^\ast$ commutes with the complex structure of
$\C P^n$ at every point.  Consequently, up to an isometry of the domain,
$\phi$ factors through the Hopf fibration as
$\phi=P\circ\pi\circ f$, where $P:\C P^n\to\C P^n$ is holomorphic.
\end{abstract}
\maketitle

\tableofcontents

\section{Introduction}\label{sec:introduction}
\subsection{Context}
Harmonic maps between Riemannian manifolds are the critical points of the
Dirichlet energy and form a nonlinear extension of several classical
geometric theories, including geodesics and minimal surfaces.  Since the
foundational work of Eells and Sampson~\cite{EellsSampson1964}, the
existence, regularity and variational properties of harmonic maps have
been studied extensively.  At a harmonic map $\phi$, the second variation
is governed by the Jacobi operator $\cJ_\phi$.  Its negative eigenspace
measures the unstable directions of the energy, and its dimension is the
Morse index of $\phi$.  For maps whose domain is a round sphere,
conformal transformations of the domain provide natural test sections
for $\cJ_\phi$ and lead to general lower bounds and rigidity questions;
see, for instance, El Soufi~\cite{ElSoufi93,ElSoufi1995}.

The complex Hopf fibrations
\[
    \pi:\bbS^{2n+1}\longrightarrow\C P^n
\]
are among the most symmetric examples of harmonic maps.  With the round
metric on the sphere and the Fubini--Study metric on the target, they are
Riemannian submersions with totally geodesic circle fibres, hence harmonic
morphisms; see~\cite{BairdWood2003}.  Their symmetry makes their Jacobi
operators accessible to harmonic analysis and representation theory.
Urakawa developed a general approach to the stability of homogeneous
harmonic maps and computed the index and nullity of the classical Hopf
map $\bbS^3\to\bbS^2$~\cite{Urakawa1987-so}.  The latter computation was
also recovered by Loubeau and Oniciuc through the spectrum of the
vertical Laplacian~\cite{LoubeauOniciuc2007}: the index is $4$ and the
nullity is $8$.

The index-$4$ value in dimension three is rigid.  Rivière proved that any
non-constant harmonic map $\phi:\bbS^3\to\bbS^2$ with
$\mathrm{Ind}(\phi)\leq4$ has index exactly $4$ and, up to an isometry of
the domain, is the composition of the Hopf fibration with a holomorphic
self-map of $\C P^1$~\cite{Riviere2023-co}.  This suggests two related
questions in higher dimensions: to determine the complete low spectrum
of the complex Hopf fibration, and to understand whether the value
$2n+2$ characterises maps which factor through it.

During the preparation of this paper, we learned of three independent
works relevant to the results obtained here.  First, Gutwein and
Langlais computed the index and nullity of the round complex Hopf
fibrations and proved the integrability of their Jacobi fields in their
study of harmonic maps with conical singularities
~\cite{GutweinLanglais2026}.  Their round-metric values agree with ours,
while our spectral result treats the full Berger family and describes
the changes of index and nullity at the two degenerate parameters.
Second, Surkov proved that, for $N>2$, a harmonic map
$\phi:\bbS^N\to\C P^m$ with $\mathrm{Ind}(\phi)=N+1$ is
pseudo-horizontally weakly conformal~\cite[Theorem~1]{Surkov2026}.  For
$N=2n+1$ and $m=n$, this is precisely the pointwise commutation relation
\[
    d\phi_x\circ d\phi_x^\ast\circ J_{\phi(x)}
    =J_{\phi(x)}\circ d\phi_x\circ d\phi_x^\ast
\]
appearing in~\eqref{eq:PHWC-commutation}; our proof of this relation was
obtained independently.  In the present setting, we use it to deduce the
factorisation $\phi=P\circ\pi\circ f$ with $P$ holomorphic.  Finally,
Chen, Wang and Zhang proved that a non-constant harmonic map
$u:\bbS^{2n+1}\to\C P^n$ has index $2n+2$ if and only if, up to an
isometry of the domain, it is the composition of the Hopf fibration with
a holomorphic self-map of $\C P^n$ of degree one
~\cite{ChenWangZhang2026}.  They also computed its nullity as
$3n^2+5n$, in agreement with the round-metric value obtained here.

\subsection{Main results}
We first express the Jacobi operator of a homogeneous fibration
$G/H\to G/K$ in terms of Casimir operators; see
Theorem~\ref{thm:Urakawa}.  This refines the approach of
Urakawa~\cite{Urakawa1987-so} and reduces the spectral problem to
branching rules and finite-dimensional representation theory.  Applying
this formula to
\[
    \SU(n+1)/\SU(n)\longrightarrow
    \SU(n+1)/\mathrm{S}(\U(1)\times\U(n))
\]
gives the complete negative and zero spectrum of
$\pi:(\bbS^{2n+1},g_\epsilon)\to(\C P^n,h)$ for $n\geq3$.  With the
normalisation in~\eqref{eq:Berger-round-normalisation}, the only
eigenvalues which can be negative are
\[
    \frac{2n}{n+1}\epsilon^2-2n
    \qquad\text{and}\qquad
    \frac{8n}{n+1}\epsilon^2-4,
\]
with real multiplicities $2(n+1)$ and $n(n+1)$, respectively.  Therefore
\begin{equation*}
    \mathrm{Ind}(\pi)=
    \begin{cases}
        (n+1)(n+2),&0<\epsilon^2<\dfrac{n+1}{2n},\\[1ex]
        2(n+1),&\dfrac{n+1}{2n}\leq\epsilon^2<n+1,\\[1ex]
        0,&\epsilon^2\geq n+1.
    \end{cases}
\end{equation*}
We also determine the nullity at every parameter; see
Theorem~\ref{thm:Hopf-negative-multiplicities}.  In particular, the unit
round metric corresponds to $\epsilon^2=(n+1)/(2n)$ and satisfies
\[
    \mathrm{Ind}(\pi)=2n+2,
    \qquad
    \mathrm{Null}(\pi)=n(3n+5).
\]

Our second main result is a higher-dimensional counterpart of Rivière's
rigidity theorem.
\begin{theorem}\label{thm:intro-low-index-rigidity}
    Let $\phi:\bbS^{2n+1}\to\C P^n$ be a harmonic map with
    $\mathrm{Ind}(\phi)=2n+2$.  Then there exist an isometry
    $f\in\Aut(\bbS^{2n+1})$ and a holomorphic map
    $P:\C P^n\to\C P^n$ such that
    \[
        \phi=P\circ\pi\circ f.
    \]
\end{theorem}
The proof first shows that equality in a pointwise linear-algebraic
inequality forces $d\phi_x\circ d\phi_x^\ast$ to commute with the complex
structure of $\C P^n$.  This makes the lowest eigenspace of
$\cJ_\phi$ complex invariant.  After conjugating the induced complex
structure on $\R^{2n+2}$ to the standard one, we obtain
$d\phi(\xi)=0$ for the Hopf vector field $\xi$.  Thus $\phi$ factors
through $\pi$, and the same eigenspace identity shows that the resulting
self-map of $\C P^n$ is holomorphic.

\subsection{Organisation}
Section~2 recalls the Jacobi operator, homogeneous Riemannian fibrations
and the representation-theoretic decomposition used to diagonalise it.
Section~3 specialises the general formula to the complex Hopf fibration
with Berger metrics and computes its Morse index and nullity.
Section~4 revisits Rivière's theorem in the coordinates used here and
then proves Theorem~\ref{thm:intro-low-index-rigidity}.
\section{Preliminaries}
\subsection{Harmonic maps and Morse index}
Let $(M^m,g)$ and $(N^n,h)$ be closed Riemannian manifolds and $\phi:M\to N$ a smooth map. If $\{e_1,\cdots, e_m\}$ is a local orthonormal frame on $M$, the energy of $\phi$ is defined as
\begin{equation}
    \mathcal{E}(\phi) = \frac{1}{2}\int_M |\mathrm{d}\phi|^2\, \mathrm{d}v_g = \frac{1}{2}\int_{M}\sum_{i=1}^m h(d\phi(e_i),d\phi(e_i))\mathrm{d}v_g,
\end{equation}
where $\mathrm{d}v_g$ is the volume form on $M$. Denoting by $E=\phi^{-1}(TN)$ the pullback bundle, any smooth variation of $\phi$ is defined by a smooth section $V\in\Gamma(E)$ by the following formula
\begin{equation}
    \phi_t(p) = \exp_{\phi(p)}(tV(p)) \qforq p\in M.
\end{equation}
If $\nabla^g$ (resp. $\nabla^h$) denotes the Levi-Civita connection on $M$ (resp. $N$), we have the induced connection $\nabla^\phi$ on $E$ defined by
\begin{equation}
    \nabla^\phi_X V = \nabla^h_{d\phi(X)}V \qforq X\in T_pM,\, V\in T_{\phi(p)}N.
\end{equation}
We say that $\phi$ is a harmonic map if it is a critical point of $\mathcal{E}$, that is, if for any smooth variation $\phi_t$ of $\phi$ we have
\begin{equation}
    \left.\frac{d}{dt}\right|_{t=0}\mathcal{E}(\phi_t) = -\int_M h(V,\tau(\phi)) = 0,
\end{equation}
where $\tau(\phi) = \sum_{i=1}^m \nabla^\phi_{e_i}d\phi(e_i) - d\phi(\nabla^g_{e_i}e_i)$ is the tension field of $\phi$. The second variation of the energy functional is given by
\begin{equation}
    \left.\frac{d^2}{dt^2}\right|_{t=0}\mathcal{E}(\phi_t) = \int_M h(V,\cJ_\phi V)\mathrm{d}v_g,
\end{equation}
where $\cJ_\phi$ is the Jacobi operator of $\phi$, defined by
\begin{equation}
    \cJ_\phi V = -\sum_{i=1}^m \nabla^\phi_{e_i}\nabla^\phi_{e_i} V - \nabla^\phi_{\nabla^g_{e_i}e_i}V - \sum_{i=1}^m R^h(V,d\phi(e_i))d\phi(e_i).
\end{equation}
Here $R^h$ is the curvature tensor of $(N,h)$, defined by $R^h(X,Y)Z = \nabla^h_X\nabla^h_Y Z - \nabla^h_Y\nabla^h_X Z - \nabla^h_{[X,Y]}Z$. The operator $\cJ_\phi$ is self-adjoint and elliptic, so its spectrum is discrete and consists of increasing sequence $\lambda_1(\cJ_\phi)<\lambda_2(\cJ_\phi)<\cdots<\lambda_n(\cJ_\phi)\nearrow\infty$ of real eigenvalues with finite multiplicity. The Morse index of $\phi$, denoted by $\mathrm{Ind}(\phi)$, is defined as the number of negative eigenvalues of $\cJ_\phi$, counted with multiplicity. The nullity of $\phi$ is defined by $\mathrm{Null}(\phi)\coloneqq\dim_\R\ker\cJ_\phi$. For later use, we fix the notation $\lambda_1(\cJ_\phi)$ for the least eigenvalue of $\cJ_\phi$, and recall that it can be computed by the Rayleigh quotient
\begin{equation}
    \lambda_1(\cJ_\phi) = \inf_{V\in\Gamma(E)}\left\{\frac{\int_M h(\cJ_\phi V,V)\mathrm{d}v_g}{\int_M h(V,V)\mathrm{d}v_g}\right\}.
\end{equation}
\subsection{Riemannian homogeneous fibrations}\label{ssec:RHF} Let $G$ be a compact connected Lie group, $H\subset K\subset G$ be closed subgroups and $\fh,\fk,\fg$ be their respective Lie algebras. We also put $s =\dim G$, $m = \dim(G/H)$ and $n = \dim(G/K)$. Let $\langle\cdot,\cdot\rangle$ be an $\Ad(G)$-invariant inner product on $\fg$, and let $\fp$ (resp. $\fq$) be the orthogonal complements of $\fh$ (resp. $\fk$) in $\fk$ (resp. $\fg$). We also define $\fm = \fp\oplus\fq$, so that we have the following orthogonal decompositions
\begin{equation}
    \fg = \fk\oplus \fq = \fh \oplus \fp\oplus \fq = \fh\oplus \fm,
\end{equation}
where there holds $\Ad(K)\fq = \fq$, $\Ad(H)\fp = \fp$ and $\Ad(H)\fm = \fm$.

Given $\epsilon,\delta>0$ real parameters, we consider the $\Ad(H)$-invariant inner product on $\fm$ given by
\begin{equation}
    \bangle{X,Y}_{\epsilon,\delta} = \frac{1}{\epsilon^2}\bangle{X_\fp,Y_\fp}+\frac{1}{\delta^2}\bangle{X_\fq,Y_\fq}, \qforq X,Y\in\fm = \fp\oplus \fq,
\end{equation}
where the subscripts denote the $\bangle{\cdot,\cdot}$-orthogonal projection onto the corresponding subspace. Notice that the second term also gives an $\Ad(K)$-invariant inner product $\delta^{-2}\bangle{\cdot,\cdot}$ on $\fq$. We shall denote by $g_{\epsilon,\delta}$ (resp. $h_\delta$) the corresponding $G$-invariant metric on $G/H$ (resp. $G/K$). With these notations, the map $\phi:(G/H,g_{\epsilon,\delta})\to (G/K,h_\delta)$, defined by $\phi(xH) = xK$, is a Riemannian submersion with totally geodesic fibers, and hence, a harmonic map.

The pullback bundle $E$ can be identified with the associated bundle $G\times_H \fq$, which is the quotient of $G\times \fq$ by the action of $H$ given by $y\cdot(x,X) = (xy^{-1},\Ad(y) X)$ for $y\in H$ and $(x,X)\in G\times\fq$. Also, the space of smooth sections $\Gamma(E)$ can be identified with the space of $H$-equivariant $\fq$-valued functions on $G$, that is,
\begin{equation}\label{def:H-eq-functions}
    C^\infty_H(G,\fq) = \{f\in C^\infty(G,\fq)\;|\; f(xy) = \Ad(y^{-1})f(x),\;\forall x\in G,\,y\in H\}.
\end{equation}
For $x\in G$, let $\tau_x$ denote the left translation induced by $x$ on
the homogeneous spaces, so that $\tau_x(yH)=xyH$ and
$\tau_x(yK)=xyK$.  The identification above is the $G$-equivariant, linear
isomorphism $\Psi:C^\infty_H(G,\fq)\to\Gamma(E)$ given by
\begin{equation}\label{eq:Psi-section-identification}
    \Psi(f)(xH)=[x,f(x)]
    =(\tau_x)_\ast f(x)\in E_{xH}.
\end{equation}
Here $[x,f(x)]$ denotes the class in $G\times_H\fq$.  The equivariance
condition in~\eqref{def:H-eq-functions} ensures that this definition is
independent of the representative of $xH$.  Explicitly, the inverse map is
given by
\begin{equation}
    \Psi^{-1}(v)(x)=(\tau_{x^{-1}})_\ast v(xH)\in T_K(G/K)\cong\fq.
\end{equation}
For each $X\in\fg$ and $f\in C^\infty_H(G,\fq)$, we define
\begin{equation}
\begin{aligned}
    Xf\coloneqq \left.\frac{d}{dt}\right|_{t=0}f\circ R_{\exp(tX)},\quad\text{ and }\quad  \ad(X)f \coloneqq \ad(X)\circ f.
\end{aligned}
\end{equation}
Notice that the first expression acts on $f$ as a derivation of the function part, and the second acts only on the target space of $f$, so these operations actually commute. For each $X\in\fg$, consider also the operator $\Theta(X):C^\infty_H(G,\fq)\to C^\infty_H(G,\fq)$ defined by
\begin{equation}
    \Theta(X)f \coloneqq Xf+\ad(X)f,
    \qforq f\in C^\infty_H(G,\fq).
\end{equation}

Fix a basis $\{X_1,\cdots, X_s\}$ of $\fg$ such that $\{X_1,\cdots, X_n\}$ is a $\delta^{-2}\bangle{\cdot,\cdot}$-orthonormal basis of $\fq$, $\{X_{n+1},\cdots, X_{m}\}$ an $\epsilon^{-2}\bangle{\cdot,\cdot}$-orthonormal basis of $\fp$ and $\{X_{m+1},\cdots, X_{s}\}$ is a $\bangle{\cdot,\cdot}$-orthonormal basis of $\fh$. For latter use, we define an $\bangle{\cdot,\cdot}$-orthonormal basis $\{Y_i\}_{i=1}^s$ of $\fg$ putting $Y_i=\delta^{-1}X_i$ for $i\in\{1,\cdots, n\}$, $Y_i = \epsilon^{-1} X_i$ for $i\in\{n+1,\cdots, m\}$ and finally $Y_i=X_i$ for $i\in\{m+1,\cdots, s\}$.
\begin{remark}
    For a section $v\in\Gamma(E)$ and $x\in G$, we shall write
    $v_{xH}$ for its value $v(xH)$.  In particular,
    $v_H=v(eH)$ denotes the value of $v$ at the identity coset.  We use the
    same convention for vector fields $Y\in\fX(G/K)$, writing
    $Y_{xK}=Y(xK)$ and, in particular, $Y_K=Y(eK)$.
\end{remark}
\begin{theorem}\label{thm:Urakawa} Let $\phi:(G/H,g_{\epsilon,\delta})\to (G/K,h_\delta)$ be the Riemannian submersion defined above. Assume that $(G/H,g_{\epsilon,\delta})$ is a geodesic orbit space and that $(G/K,h_\delta)$ is a symmetric space. Then, under the identification~\eqref{eq:Psi-section-identification}, we have the following expression for the Jacobi operator $\cJ_\phi$:
\begin{equation}
    \cJ_\phi = \delta^2(D_K-D_G+D_G^\ad-D^\ad_K)-\epsilon^2D_\Theta,
\end{equation}
where the operators $D_\ast$ are defined as follows.
\begin{equation}\label{eq:jacobian-pieces}
\begin{aligned}
    D_G  &= \sum_{k=1}^{s} Y_k^2&&
    D^\ad_G  = \sum_{k=1}^{s} \ad(Y_k)^2\\
    D_K &= \sum_{\mathclap{k=n+1}}^sY_k^2&&
    D^\ad_K = \sum_{\mathclap{k=n+1}}^s\ad(Y_k)^2\\
    D_\Theta &= \sum_{k=n+1}^s\Theta(Y_k)^2
\end{aligned}
\end{equation}
\end{theorem}
\begin{proof}
    In this proof, we follow the preparations in~\cite[8.2]{Urakawa1987-so}, but with some different notations. Choose a sufficiently small neighborhood $U$ of the identity in $G$ and set
    \begin{equation}
        N\coloneqq U\cap\exp(\fq),
        \qquad
        N_K\coloneqq U\cap\exp(\fp).
    \end{equation}
    We take $U$ small enough that multiplication defines a diffeomorphism
    $N\times N_K\to N\cdot N_K$, and that the restrictions of the quotient
    maps $\pi_K:G\to G/K$ and $\pi_H:G\to G/H$ define diffeomorphisms from
    $N$ and $N\cdot N_K$ onto neighborhoods of $K$ and $H$, respectively.
    Here $N\cdot N_K=\{yk:y\in N,\ k\in N_K\}$.  On the neighborhood $\pi_H(N\cdot N_K)$ of $H\in G/H$ we define for $X\in\fm$
    \begin{equation}\label{eq:good-frame}
        X^\ast_{xH} \coloneqq (\tau_{x})_\ast X\in T_{xH}G/H,
        \qforq x\in N\cdot N_K.
    \end{equation}
    For $X\in \fq$ we also define
    \begin{equation}
        \Tilde{X}_{yK} \coloneqq (\tau_{y})_\ast X\in T_{yK}G/K,
        \qforq y\in N.
    \end{equation}
    Now, for $f\in C^\infty_H(G,\fq)$ we write $f(x) = f^i(x)Y_i,\;x\in G$. Then, if $v = \Psi(f)$ and $x=yk$ with $y\in N$ and $k\in N_K$, we have
    \begin{equation}
        v(xH) = \Psi(f)(xH) = f^i(x)(\tau_x)_\ast Y_i = f^i(yk)(\tau_y)_\ast \Ad(k)Y_i,
    \end{equation}
    since the isotropy action on $\fq$ is equivalent to the adjoint action, and writing $\Ad(k)Y_i = a_i^j(k)Y_j$ gives us
    \begin{equation}\label{eq:perfect-for-pullback}
        v(xH) = f^i(yk)a_i^j(k)(\tau_y)_\ast Y_j = \Tilde{f}^j(ykH)\phi^\ast(\Tilde{Y}_j),
    \end{equation}
    where $\Tilde{f}^j(ykH) \coloneqq f^i(yk)a_i^j(k)$. It follows that equation~\eqref{eq:perfect-for-pullback} is the suitable expression of $v\in\Gamma(E)$ to compute the Jacobian operator. From~\cite[Lemma~8.15]{Urakawa1987-so} we get, together with the assumption of $G/K$ being symmetric irreducible, that
    \begin{equation}
        (\nabla^\phi _{X^\ast}\nabla^\phi_{X^\ast}v)_H = X^\ast_{H}(X^\ast \Tilde{f}^j)(\Tilde{Y}_j)_{K}, \qforq X\in\fm.
    \end{equation}
    Also, applying~\cite[Lemma~8.6]{Urakawa1987-so} we get
    \begin{small}
    \begin{equation}
    \begin{aligned}
    X^\ast_{H}(X^\ast \Tilde{f}^j) &= (X^2f^j)(e) + 2(Xf^i)(e)([X,Y_i],Y_j)+f^i(e)([X,[X,X_i]],X_j), \qforq X\in\fp\\
    X^\ast_{H}(X^\ast \Tilde{f}^j) &= (X^2f^j)(e), \qforq X\in\fq.
    \end{aligned}
    \end{equation}
    \end{small}
    Multiplying by $Y_j$ on both sides and summing over $j$, we get
    \begin{equation}
    \begin{aligned}
        \Psi^{-1}\bigl(\nabla^\phi_{X^\ast}\nabla^\phi_{X^\ast}v\bigr)(e)
        &= (X^2f)(e)+2\bigl((X\circ\ad(X))f\bigr)(e)
        +(\ad(X)^2f)(e),
        \qforq X\in\fp,\\
        \Psi^{-1}\bigl(\nabla^\phi_{X^\ast}\nabla^\phi_{X^\ast}v\bigr)(e)
        &= (X^2f)(e),
        \qforq X\in\fq.
    \end{aligned}
    \end{equation}
    To finally obtain the Laplacian $\roughtil v$ we need also to compute $(\nabla^\phi_{\nabla^g_{X^\ast}X^\ast}v)_H$. For this, we recall that
    \begin{equation}
        (\nabla^g_{X^\ast} Y^\ast)_H = \frac{1}{2}[X,Y]_\fm + U(X,Y), \qforq X,Y\in M,
    \end{equation}
    where $U$ is defined implicitly by
    \begin{equation}
        2\bangle{U(X,Y),Z}_{\epsilon,\delta} = \bangle{[Z,X]_\fm,Y}_{\epsilon,\delta} + \bangle{X,[Z,Y]_\fm}_{\epsilon,\delta}.
    \end{equation}
    Now, since we assume $(G/H,g_{\epsilon,\delta})$ to be geodesic orbit, we have $U(X,X) =0$ for all $X\in\fm$, meaning that
    \begin{equation}
        (\nabla^g_{X^\ast} X^\ast)_H = 0, \qforq X\in\fm,
    \end{equation}
    So $(\nabla^{\phi}_{\nabla^g_{X^\ast}X^\ast}v)_H = 0$. Finally, since $\{X_i^\ast\}_{i=1}^m$ is an orthonormal frame for $g_{\epsilon,\delta}$ on $\pi_H(N\cdot N_K)$, we conclude
    \begin{equation}
        \Psi^{-1}(\roughtil v)(e)
        =-\sum_{i=1}^m(X_i^2f)(e)
        -2\sum_{i=n+1}^m\bigl((X_i\circ\ad(X_i))f\bigr)(e)
        -\sum_{i=n+1}^m(\ad(X_i)^2f)(e).
    \end{equation}
    On the other hand, since $(G/K,h_\delta)$ is a symmetric space, its
    curvature tensor at the origin is given by
    \begin{equation}
        R^h(\Tilde{X},\Tilde{Y})\Tilde{Z}_K
        =-[[X,Y],Z], \qforq X,Y,Z\in\fq.
    \end{equation}
    Therefore, the curvature term of the Jacobi operator satisfies
    \begin{equation}
        \Psi^{-1}\left(\sum_{i=1}^n
        R^h\bigl(v,d\phi(X_i^\ast)\bigr)d\phi(X_i^\ast)\right)(e)
        =-\sum_{i=1}^n(\ad(X_i)^2f)(e).
    \end{equation}
    Putting all this together, in terms of the basis $\{Y_i\}$ of $\fg$, we get globally
    \begin{small}
    \begin{equation}\label{eq:big-J}
    \begin{aligned}
        \Psi^{-1}\circ\cJ_\phi\circ\Psi
        =&\;\delta^2(-D_G+D_K)
        +\epsilon^2\left(-D_K+\sum_{i=m+1}^sY_i^2\right)\\
        &+2\epsilon^2\left(-\sum_{i=n+1}^sY_i\circ\ad(Y_i)
        +\sum_{i=m+1}^sY_i\circ\ad(Y_i)\right)\\
        &+\epsilon^2\left(-D^\ad_K+\sum_{i=m+1}^s\ad(Y_i)^2\right)
        +\delta^2(D^\ad_G-D^\ad_K).
    \end{aligned}
    \end{equation}
    \end{small}
    Now, for $X\in\fh$ we have
    \begin{equation}
        Xf =  \left.\frac{d}{dt}\right\rvert_{t=0}f\circ R_{\exp(tX)} = \ddto \Ad(\exp(-tX))f = -\ad(X)f,
    \end{equation}
    Thus, as operators on $C^\infty_H(G,\fq)$, for
    $i\in\{m+1,\cdots,s\}$ we can write
    \begin{equation}\label{eq:squares-h}
         Y_i^2 = \ad(Y_i)^2 = - Y_i\circ\ad(Y_i).
    \end{equation}
    Also, for each $X\in\fg$ there holds $\Theta(X)^2 = X^2 + 2X\circ\ad(X) + \ad(X)^2$
    and as a consequence
    \begin{equation}\label{eq:D_theta}
        D_\Theta = D_K + D^\ad_K + 2\sum_{i=n+1}^sY_i\circ \ad(Y_i)
    \end{equation}
    To obtain~\eqref{eq:jacobian-pieces} we just plug equations~\eqref{eq:squares-h} and~\eqref{eq:D_theta} into equation~\eqref{eq:big-J}.
    \end{proof}

\subsection{Representation theory}
Let $G$ be a compact connected Lie group with Lie algebra $\mathfrak g$.
We fix a maximal torus $T\subset G$ and denote by
$\mathfrak t_{\mathbb C}\subset \mathfrak g_{\mathbb C}$ the corresponding
Cartan subalgebra of the complexified Lie algebra.  Let
$\Delta_G\subset\ft_\C^\ast$ be the corresponding root system and fix a
set of positive roots $\Delta_G^+$.  For each $\alpha\in\Delta_G$, let
\begin{equation}
    (\fg_\C)_\alpha
    =\{X\in\fg_\C:[H,X]=\alpha(H)X\ \text{for every }H\in\ft_\C\}
\end{equation}
be the associated root space.  It is well known that each root space is
one-dimensional; thus, for each
$\alpha\in\Delta_G$, we may choose $X_\alpha\in\fg_\C$ such that
$(\fg_\C)_\alpha=\C X_\alpha$.  We denote by
$\rho_G=\frac{1}{2}\sum_{\alpha\in\Delta_G^+}\alpha$ the half-sum of the
positive roots.

Given a complex representation $\pi:G\to\GL(\cH)$, over a complex vector space $\cH$, we consider its differential $\pi_\ast:\fg_\C\to\gl(\cH)$, which is a representation of $\fg_\C$, and define, for each $\mu\in\ft_\C^\ast$, the corresponding weight space

\[\cH(\mu) \coloneq \{v\in \cH\;|\; \pi_\ast(X)v = \mu(X)v, \quad \forall\,X\in\ft_\C\}.\]
The set of weights of $\pi$ is then defined as
\[\Lambda(\cH) \coloneqq \{\mu\in \mathfrak{t}_\C^\ast\;|\;\cH(\mu) \neq 0\},\]
and since the action of $\pi_\ast(\ft_\C)$ on $\cH$ is diagonalizable we have $\cH = \bigoplus_{\mu\in\Lambda(\cH)}\cH(\mu)$. If $\pi$ is irreducible, then there is a unique  $\lambda\in\Lambda(\cH)$ satisfying $\pi_\ast(X_\alpha)\cH(\lambda) = 0$ for all $\alpha\in\Delta^+_G$. The weight $\lambda$ is called the highest weight of the irreducible representation $\pi$. The set $\Lambda^+_G$ of all the possible highest weights of irreducible representations of $G$ is called the set of dominant integral weights, and there is a bijective correspondence between this set and the isomorphism classes of irreducible representations of $G$. For each $\lambda\in\Lambda^+_G$, we denote the corresponding irreducible representation by $\pi^G_\lambda:G\to\GL(\cH^G_\lambda)$.

We fix an $\Ad(G)$-invariant inner product $\bangle{\cdot,\cdot}$ on $\fg$, and use the same notation for the induced bilinear form on $\ft^\ast_\C$. For a representation $\pi:G\to\GL(\cH)$ and an orthonormal basis $\{X_i\}_{i=1}^s$ of $\fg$ with respect to $\bangle{\cdot,\cdot}$, we define the Casimir operator $\Cas_G^\pi:\cH\to \cH$ by
\[\Cas_G^\pi = \sum_{i=1}^s \pi_\ast(X_i)^2.\]
If $\pi=\pi^G_\lambda$ is irreducible, then
\begin{equation}
    \Cas_G^{\pi} = - \bangle{\lambda+2\rho_G,\lambda}\id_{\cH}.
\end{equation}
The number $\bangle{\lambda+2\rho_G,\lambda}$ will be denoted by $c_G(\lambda)$ and called the Casimir eigenvalue of $G$ on $\pi^G_\lambda$.

Now, let $H\subset G$ be a closed subgroup and $\tau:H\to\GL(\cW_\tau)$ be a complex representation of $H$. The space of $H$-equivariant functions $C^\infty_H(G,\cW_\tau)$, defined analogously as in~\eqref{def:H-eq-functions}, is the representation space of the induced $G$-representation $\Ind_H^G\tau:G\to\GL(C^\infty_H(G,\cW_\tau))$ defined by
\begin{equation}
    (\Ind_H^G\tau)(a)f = f\circ L_{a^{-1}}, \qforq a\in G,\, f\in C^\infty_H(G,\cW_\tau).
\end{equation}
The representation $\Ind_H^G\tau$ is also called the left regular representation, it is a infinite-dimensional representation of $G$, and its decomposition into irreducible $G$-representations is given by
\begin{equation}
    C^\infty_H(G,\cW_\tau) \cong \bigoplus_{\lambda\in\Lambda_G^+} \mathcal H_\lambda^G\otimes \Hom_H(\mathcal H_\lambda^G,\cW_\tau).
\end{equation}
For each dominant weight $\lambda\in\Lambda^+_G$, we call the space $\cM^\tau_\lambda \coloneqq \mathcal H_\lambda^G\otimes \Hom_H(\mathcal H_\lambda^G,\cW_\tau)$ the $\lambda$-isotypical component of $C^\infty_H(G,\cW_\tau)$ and $m_\lambda^\tau \coloneqq \dim_\C \Hom_H(\mathcal H_\lambda^G,\cW_\tau)$ its multiplicity. The inclusion of $\cM^\tau_\lambda$ into $C^\infty_H(G,\cW_\tau)$ is given by the mapping $v\otimes P\mapsto f_{v,P}$, where for $x\in G$, we define
\begin{equation}
    f_{v,P}(x) \coloneqq P(\pi^G_\lambda(x)^{-1}v).
\end{equation}
This discussion is useful because of the following result, which is classical for spectral theory of homogeneous spaces, and identifies the action of the Jacobian operator $\cJ_\phi$ with the action of a few Casimir operators.

Before continuing, we recall that if $\pi:G\to\GL(\cH)$ is a real irreducible representation, then it can be of real, complex and quaternionic type, see~\cite[Chapter II]{Brocker2010-sv}. If $\pi$ is of real type, then $(\pi,\cH_\C)\coloneqq(\tau,\cV)$ is complex irreducible. If $\pi$ is of complex type, then there exists a complex irreducible representation $(\tau,\cV)$ such that $\cH_\C = \cV\oplus\cV^\ast$. Finally, if $\pi$ is of quaternionic type, then there exists a complex irreducible representation $(\tau,\cV)$ such that $\cH_\C = \cV\oplus\cV$. In any case, we shall define the highest weight of $\pi$ to be the highest weight of the associated complex irreducible representation $(\tau,\cV)$.
\begin{lemma}\label{prop:Ds}
    Under the notation of Theorem~\ref{thm:Urakawa}, denote by $(\fq,\Ad_\fq)$ the isotropy representation of $G/K$. Assume also that $G$ is simple. For each $\lambda\in\Lambda^+_G$, and $f_{v,P}\in\cW^{\Ad_\fq}_\lambda \coloneqq \cW_\lambda$, the following hold:
    \begin{enumerate}
        \item $D_G(f_{v,P})= f_{\Cas_G^\lambda v,P}$;
        \item $D_K(f_{v,P}) = f_{v,P\circ \Cas_K^{\lambda}}$;
        \item $D_\Theta(f_{v,P})= f_{v,\Cas_K^{\lambda,\fq}(P)}$ and
        \item $(D^\ad_G-D^\ad_K)(f_{v,P}) = (c_K(\mu_{\fq}) - c_G(\mu_{\fg})) f_{v,P}$, where $\mu_\fq$ ( resp. $\mu_\fg$) is the highest weight of $\Ad_\fq$ (resp. $\Ad$).
    \end{enumerate}
    Here, $\Cas^\lambda_K$ (resp. $\Cas^{\lambda,\fq}_K$) is Casimir operator of the representation $\pi^G_\lambda\rvert_{K}$ (resp. $(\pi^G_\lambda)^\ast\otimes\Ad_\fq$).
\end{lemma}
\begin{proof}
    The first three items follow directly from the fact that, for any $X\in \fg$ we have
    \begin{equation}\label{eq:action-on-eq-functions}
        \begin{aligned}
            (Xf_{v,P})(x) = \left.\frac{d}{dt}\right\rvert_{t=0} P(\pi^G_\lambda(x\exp(tX))^{-1}v) =& \left.\frac{d}{dt}\right\rvert_{t=0} P(\pi^G_\lambda(\exp(-tX))\pi^G_\lambda(x)^{-1}v)\\
            =&\,- P((\pi_\lambda^G)_\ast(X)\pi_\lambda^G(x)^{-1}v).
        \end{aligned}
    \end{equation}
    Indeed, from equation~\eqref{eq:action-on-eq-functions} we can prove the second item, and also
    \begin{equation*}
        D_G(f_{v,P})(x) = P(\Cas_G^\lambda\pi_\lambda^G(x)^{-1}v) = P(\pi_\lambda^G(x)^{-1}\Cas_G^\lambda v) = f_{\Cas_G^\lambda v,P}.
    \end{equation*}
    because $\Cas_G^\lambda$ commutes with $G$-representations.
    Finally, since $G$ is simple, $\fg$ is an irreducible representation 
    \begin{equation*}
        D^\ad_{G}f_{v,P} = \Cas_G^{\Ad}\circ f_{v,P} = - c_G(\mu_\fg) f_{v,P}.
    \end{equation*} Analogously, $G/K$ being irreducible means that $\fq$ is irreducible, so
    \begin{equation*}
        D^\ad_{K}f_{v,P} = \Cas_K^{\ad_\fq}\circ f_{v,P} = - c_K(\mu_\fq) f_{v,P}.\qedhere
    \end{equation*}
\end{proof}
\begin{definition}\label{def:C-lambda}
    Let $\lambda\in\Lambda_G^+$ and let $\mu_1,\mu_2\in\Lambda_K^+$. We define
    \[
    C_{\lambda}(\mu_1,\mu_2)\coloneqq \delta^2\Big(c_G(\lambda)-c_K(\mu_1)+c_K(\mu_{\fq})-c_G(\mu_{\fg})\Big)+\epsilon^2c_K(\mu_2).
    \]
\end{definition}

\begin{theorem}\label{thm:J-eigenvalues}
    Let $H\subset K\subset G$ be compact Lie groups, such that $G$ is simple and $G/K$ is an irreducible symmetric space with orthogonal decompositions $\fg=\fk\oplus\fq$ and $\fk=\fh\oplus\fp$, w.r.t an $\Ad(G)$-invariant inner product $\bangle{\cdot,\cdot}$ on $\fg$. Define the inner products $\delta^{-2}\bangle{\cdot,\cdot}\rvert_{\fq}$ on $\fq$ and $\epsilon^{-2}\bangle{\cdot, \cdot}\rvert_{\fp}$ on $\fp$, for some $\delta,\epsilon>0$, inducing $G$-invariant metrics $g_{\epsilon,\delta}$ (resp. $h_{\delta}$) on $G/H$ (resp. $G/K$). Let $\phi:(G/H,g_{\epsilon,\delta})\to (G/K,h_\delta)$ be the harmonic map defined by $\phi(xH) = xK$, and denote by $\cJ_\phi$ its Jacobi operator. Fixing a maximal Abelian subalgebra $\ft\subset \fg$, we set the following notations:
    \begin{enumerate}
        \item $\mu_\fg\in\Lambda_G^+$ (resp. $\mu_\fq\in\Lambda_K^+$) denotes the highest $G$-weight (resp. $K$-weight) of the adjoint representation $\Ad$ (resp. the isotropy representation $\Ad_\fq$). Also let $(\pi_{\fq},\cV_\fq)$ be the associated complex irreducible representation;
        \item For $\lambda\in\Lambda_G^+$, let $\pi^G_\lambda = \pi^K_{\lambda_1}\oplus\cdots\oplus\pi^K_{\lambda_d}$ be the decomposition as $K$-irreducible representations of $\pi^G_\lambda$;
        \item For each $i\in\{1,\cdots,d\}$, consider the decompositions $(\pi^K_{\lambda_i})^\ast\otimes\pi_\fq = \pi^K_{\lambda_{i1}}\oplus \cdots\oplus\pi^K_{\lambda_{ir_i}}$ and $(\pi^K_{\lambda_i})^\ast\otimes(\pi_\fq)^\ast = \pi^K_{\bar{\lambda}_{i1}}\oplus \cdots\oplus\pi^K_{\bar{\lambda}_{is_i}}$ into $K$-irreducible representations.
    \end{enumerate}
    Then the following table gives the eigenvalues of $\cJ_\phi$, with the corresponding multiplicities, depending on the type of $\Ad_\fq$. We use the notations $d_\lambda = \dim_\C\cH^G_\lambda$, $m_{\lambda_{ij}}=\dim_\C\Hom_H(\cH^K_{\lambda_{ij}},\C)$ and $m_{\bar{\lambda}_{ij}}=\dim_\C\Hom_H(\cH^K_{\bar{\lambda}_{ij}},\C)$.
    \begin{table}[ht]
    \centering
    \begin{tabular}{lcc}
    \toprule
    Type of $\Ad_{\fq}$ & Eigenvalues & Multiplicities \\[1ex]
    \midrule
    Real          & $C_{\lambda}(\lambda_i,\lambda_{ij})$ & $m_{\lambda_{ij}}d_\lambda$ \\[1ex]
    Complex       & $C_{\lambda}(\lambda_i,\lambda_{ij}),\, C_{\lambda}(\lambda_i,\bar{\lambda}_{ij})$ & $m_{\lambda_{ij}}d_\lambda$,  $m_{\bar{\lambda}_{ij}}d_\lambda$ \\[1ex]
    Quaternionic  & $C_{\lambda}(\lambda_i,\lambda_{ij})$ & $2m_{\lambda_{ij}}d_\lambda$ \\[1ex]
    \bottomrule
    \end{tabular}
    \label{tab:adfp-types}
    \end{table}

\end{theorem}
\begin{proof}
    For $\lambda\in\Lambda_G^+$, we have
    \begin{equation*}
        \Hom_H(\cH^G_\lambda,\fq) = \begin{cases}\displaystyle\bigoplus_{j=1}^{r_i}\bigoplus_{i=1}^d (\cH^K_{\lambda_{ij}})_0, &\text{ if }\fq\text{ is of real type};\\
        \displaystyle\bigoplus_{j=1}^{r_i}\bigoplus_{i=1}^d (\cH^K_{\lambda_{ij}})_0\oplus (\cH^K_{\bar{\lambda}_{ij}})_0, &\text{ if }\fq\text{ is of complex type};\\
        \displaystyle\bigoplus_{j=1}^{r_i}\bigoplus_{i=1}^d (\cH^K_{\lambda_{ij}})_0\oplus (\cH^K_{\lambda_{ij}})_0, &\text{ if }\fq\text{ is of quaternionic type};\\
        \end{cases}
    \end{equation*}
    where for a $K$-representation $(\pi,\cH)$, the notation $(\cH)_0$ stands for its $H$-trivial part. Now, take $f_{v,P}$ such that $v\in\cH^G_\lambda$ and $P\in(\cH^K_{\lambda_{ij}})_0\subset \cH^K_{ij}\subset \Hom_H(\cH^K_{\lambda_i},\fq)$. We compute, with help of Lemma~\ref{prop:Ds}
    \begin{equation*}
        \begin{aligned}
            D_G(f_{v,P}) &= f_{\Cas^\lambda_Gv,P} = - c_G(\lambda) f_{v,P}\\
            D_\Theta(f_{v,P}) &= f_{v,\Cas_K^{\lambda,\fq}(P)} = -c_K(\lambda_{ij})f_{v,P}.
    \end{aligned}
    \end{equation*}
    Also, for $x\in G$, decompose $\pi^G_\lambda(x)^{-1}v = \sum_{j=1}^d w_i$ for $w_j\in\cH^K_{\lambda_j}$. Since $P\in \Hom_H(\cH^K_{\lambda_i},\fq)$, we have $f_{v,P}(x) = P(w_i)$. On the other hand, Lemma~\ref{prop:Ds} again gives us
    \begin{equation}
        (D_K f_{v,P})(x) = \sum_{j=1}^dP\left(\Cas_K^{\lambda}w_j\right) = -\sum_{j=1}^dc_K(\lambda_j)P\left(w_j\right) = - c_K(\lambda_i)P(w_i),
    \end{equation}
    and therefore $(D_K f_{v,P}) = -c_K(\lambda_i)f_{v,P}$, since it is valid for any $x\in G$. The equality $(D^\ad_G-D^\ad_K)f_{v,P} = (c_K(\mu_\fq)-c_G(\mu_\fg))f_{v,P}$ follows directly from~\ref{prop:Ds}, Item 4. These computations together with~\ref{thm:Urakawa} imply that
    \begin{equation}
        \cJ_\phi(f_{v,P}) = C_{\lambda}(\lambda_i,\lambda_{ij})f_{v,P}.
    \end{equation}
    This proves that $\cH^G_\lambda\otimes(\cH^K_{\lambda_{ij}})_0$ is an eigenspace of $\cJ_\phi$, for the eigenvalue $C_{\lambda}(\lambda_i,\lambda_{ij})$, with dimension $m_{\lambda_{ij}}d_\lambda$. This fully describe the spectrum of $\cJ_\phi$ for the case $\Ad_\fq$ is of real type. If $\Ad_\fq$ is of complex type, then an analogous argument using $P\in(\cH^K_{\bar{\lambda}_{ij}})_0$ gives the eigenvalues $C_{\lambda}(\lambda_i,\bar{\lambda}_{ij})$ with the corresponding multiplicities. Finally, if $\Ad_\fq$ is of quaternionic type, then all the eigenspaces are doubled, because $\fq_\C = \cV_\fq\oplus\cV_\fq$, hence all the multiplicities are doubled.
\end{proof}
\begin{remark}\label{rem:real-vs-complex-multiplicity}
    Suppose that $\fq$ is of complex type and write
    $\fq_\C=\cV_\fq\oplus\overline{\cV_\fq}$. There is a useful symmetry
    between the two summands. If $\lambda\in\Lambda^+_G$, then, following
    the notation of Theorem~\ref{thm:J-eigenvalues},
    \begin{equation}
        \Hom_H(\cH^G_\lambda,\fq_\C)
        \cong \Hom_H(\cH^G_{\lambda^\ast},\fq_\C)^\ast,
    \end{equation}
    where $\lambda^\ast\in\Lambda^+_G$ is the highest weight of the dual
    representation. In particular,
    $\lambda_{ij}=\overline{\lambda^\ast}_{ij}$ and
    $\lambda^\ast_{ij}=\overline{\lambda}_{ij}$, so the corresponding
    eigenvalues of the Jacobi operator agree and have the same multiplicity.
    Consequently, it is enough to consider the weights satisfying
    $\Hom_H(\cH^G_\lambda,\cV_\fq)\neq0$: they give the complete real
    spectrum, and every multiplicity obtained from this chosen complex
    summand must be doubled. Thus the eigenvalue associated with
    $\lambda_{ij}$ has real multiplicity
    $2m_{\lambda_{ij}}d_\lambda$.
\end{remark}
\section{Morse index of higher dimensional complex Hopf fibration}
In this section, we compute the Morse index and the nullity of the complex Hopf fibration $\phi:(\bbS^{2n+1},g_{\epsilon})\to(\C P^n,h)$, where $g_\epsilon$ and $h$ are the $\epsilon$-Berger metric and the Fubini-Study metric, respectively. We will follow the notations of~\cite{Ikeda78}. Letting $G = \SU(n+1)$, $H=\SU(n)$ and
\[K=\mathrm{S}(\U(1)\times \U(n)) = \left\{\begin{pmatrix}\det(A)^{-1}&0\\0&A\end{pmatrix}\in \GL(n+1,\C)\;|\; A\in \U(n)\right\},\]
Then $G/H\cong \bbS^{2n+1}$, $G/K\cong \C P^n$ and the Hopf map $\phi$ is just the projection $xH\mapsto xK$ we studied in \S~\ref{ssec:RHF}. Following the notations of the previous sections we have
\begin{equation}
    \begin{aligned}
        \fg &= \left\{X\in \gl(n+1,\C)\;|\; X+X^\ast = 0,\,\tr(X) = 0\right\} && \dim_\R \fg = n^2+2n\\
        \fk &= \left\{\begin{pmatrix}
            ia&0\\0&A
        \end{pmatrix}\;|\;A\in \fu(n)\text{ and }ia+\tr A = 0\right\} && \dim_\R \fk = n^2\\
        \fh &= \left\{\begin{pmatrix}
            0&0\\0&A
        \end{pmatrix}\;|\;A\in \su(n)\right\} && \dim_\R \fh = n^2-1\\
        \fq &= \left\{\begin{pmatrix}
            0&-v^\ast\\v&0
        \end{pmatrix}\;|\;v\in\C^n\right\} && \dim_\R \fq = 2n\\
        \fp &= \left\{\begin{pmatrix}
            ia&0\\0&-\frac{ia}{n}I
        \end{pmatrix}\;|\;a\in\R\right\} && \dim_\R \fp = 1.
    \end{aligned}
\end{equation}
Our $\Ad$-invariant inner product on $\fg$ will be given by
\[\bangle{X,Y} = -\frac{1}{2}\tr(XY) \qforq X,Y\in\fg.\]
With the notations of~\ref{thm:J-eigenvalues}, we fix $\epsilon>0$ and $\delta = 1$, and denote the corresponding $SU(n)$-invariant metrics by $g_\epsilon$ and $h$.
In this normalisation, let $g_{\mathrm{rd}}$ be the unit round metric, $\xi$ its unit Hopf vector field, and put $\eta=g_{\mathrm{rd}}(\xi,\cdot)$. The horizontal metric agrees with $g_{\mathrm{rd}}$, while the vertical generator $Z=\diag(i,-i/n,\ldots,-i/n)$ satisfies $\bangle{Z,Z}=(n+1)/(2n)$. Hence
\begin{equation}\label{eq:Berger-round-normalisation}
    g_\epsilon=g_{\mathrm{rd}}+\left(\frac{n+1}{2n\epsilon^2}-1\right)\eta\otimes\eta,
\end{equation}
so $g_\epsilon$ is the unit round metric precisely when $\epsilon^2=(n+1)/(2n)$.


As a maximal abelian sub-algebra of $\fg$ we choose
\begin{equation}
    \ft = \{\diag(it_1,\cdots, it_{n+1})\;|\; t_1+\cdots+t_{n+1}= 0\},
\end{equation}
and define $e_i\in\ft^\ast$ by $e_j(\diag(it_1,\cdots, it_{n+1})) = t_j$. Complexifying, we have $\fg^\C = \fsl(n+1,\C)$ and the set of positive roots is $\Delta^+_G = \{e_i-e_j\;|\;i<j\}$. We then have
\begin{equation}
    \begin{aligned}
        \Lambda^+_G &= \left\{\sum_{j=1}^{n}a_je_j\;|\; a_1\geq a_2\geq \cdots \geq a_{n}\geq 0\right\}\\
        \Lambda^+_K &= \left\{\sum_{j=1}^{n}a_je_j\;|\; a_2\geq a_3\geq \cdots \geq a_{n}\geq 0\right\}.
    \end{aligned}
\end{equation}
The isotropy representation $(\Ad_\fq,\fq)$ is of complex type, due to the complex structure on $\C P^n$, and then its complexification decomposes into $K$-irreducible pieces as
\begin{equation}
    \Ad_\fq = \pi^K_{-e_1+e_2}\oplus \pi^K_{2e_1+e_2+\cdots+e_n}.
\end{equation}
This is the algebraic counterpart of the Hodge decomposition. We choose
$\pi_\fq=\pi^K_{-e_1+e_2}$ as the complex irreducible representation
associated with $\Ad_\fq$ and determine the highest weights
$\lambda\in\Lambda_G^+$ for which
\begin{equation}\label{eq:contributing-component}
    \Hom_H(\cH^G_{\lambda},\cH^K_{-e_1+e_2})\neq0.
\end{equation}
The conjugate summand $\pi_\fq^\ast=\pi^K_{2e_1+e_2+\cdots+e_n}$
will be recovered later by complex conjugation and will not be included in
the list of weights below.
For this we use the classical branching rule for the restriction from $G=\SU(n+1)$ to $K=\mathrm{S}(\U(1)\times \U(n))\cong\U(n)$.
\begin{proposition}[{\cite{Ikeda78}}]\label{thm:Ikeda} Let $\lambda = \sum_{i=1}^n a_je_j\in\Lambda_G^+$ and $\mu = \sum_{j=1}^nb_je_j\in\Lambda^+_K$. Then $\pi_\mu^K$ is a $K$-irreducible component of $\pi_\lambda^G$ if and only if there exists $q\in\Z$ such that
    \begin{equation}\label{eq:branching}
        \begin{aligned}
            &a_1\geq b_2+q\geq a_2\geq b_3+q\geq \cdots \geq a_{n-1}\geq b_n+q\geq a_n\geq q\geq 0,\\
            &b_1=\sum_{j=1}^n a_j - \sum_{j=2}^n b_j - (n+1)q.
        \end{aligned}
    \end{equation}
    In this case, the multiplicity $m_\lambda^\mu$ is $1$.
\end{proposition}
\begin{lemma}\label{lemma:Hopf-branching}
    Assume $n\geq3$. The weights $\lambda\in\Lambda_G^+$ satisfying
    \begin{equation*}
        \Hom_H(\cH^G_{\lambda},\cH^K_{-e_1+e_2})\neq0
    \end{equation*}
    are exactly the weights
    \begin{equation}\label{eq:possible-lambdas}
        \begin{aligned}
            \lambda^{(1)}(a,k)
            &\coloneqq(2a-k-1)e_1+(a-k-1)(e_2+\cdots+e_n),
            &&a\geq1,\quad k\leq a-1,\\
            \lambda^{(2)}(a,k)
            &\coloneqq(2a-k)e_1+(a-k+1)e_2+(a-k)(e_3+\cdots+e_n),
            &&a\geq1,\quad k\leq a,
        \end{aligned}
    \end{equation}

    Write $\mu^{(1)}(\ell)=\ell e_1+e_2$ and
    $\mu^{(2)}(\ell)=\ell e_1+e_2+\cdots+e_n$. For the two families in
    \eqref{eq:possible-lambdas}, the $K$-irreducible components whose
    restriction to $H$ can contribute to $\Hom_H(\cH^G_\lambda,\fq_\C)$ are
    \begin{equation}\label{eq:relevant-K-components}
        \begin{aligned}
            \left.\pi^G_{\lambda^{(1)}(a,k)}\right|_K
            &=\tau^{(1)}\oplus\pi^K_{\mu^{(1)}(k)}\oplus
            \begin{cases}
                \pi^K_{\mu^{(2)}(k+3)},&k\leq a-2,\\
                0,&k=a-1,
            \end{cases}\\[1ex]
            \left.\pi^G_{\lambda^{(2)}(a,k)}\right|_K
            &=\tau^{(2)}\oplus\pi^K_{\mu^{(1)}(k)}\oplus
            \begin{cases}
                \pi^K_{\mu^{(2)}(k-1)},&n=3,\\
                0,&n\geq4.
            \end{cases}
        \end{aligned}
    \end{equation}
    Here $\tau^{(j)}$ is a $K$-representation satisfying
    $\Hom_H(\cH_{\tau^{(j)}},\fq_\C)=0$, and every displayed irreducible
    summand has multiplicity one.

    In the notation of Theorem~\ref{thm:J-eigenvalues}, choose
    $\pi_\fq=\pi^K_{\mu^{(1)}(-1)}$, so that
    $\pi_\fq^\ast=\pi^K_{\mu^{(2)}(2)}$. For every $k\in\Z$, the corresponding
    tensor products decompose into $K$-irreducible representations as follows:
    \begin{equation}\label{eq:K-tensor-decompositions}
        \begin{aligned}
            (\pi^K_{\mu^{(1)}(k)})^\ast\otimes\pi_\fq
            &=\pi^K_{-(k+1)e_1}\oplus\pi^K_{-ke_1+2e_2+e_3+\cdots+e_n},\\
            (\pi^K_{\mu^{(1)}(k)})^\ast\otimes\pi_\fq^\ast
            &=\pi^K_{(3-k)e_1+2(e_2+\cdots+e_n)}\oplus\pi^K_{(2-k)e_1+e_2+\cdots+e_{n-1}},\\
            (\pi^K_{\mu^{(2)}(k)})^\ast\otimes\pi_\fq
            &=\pi^K_{-ke_1+2e_2}\oplus\pi^K_{-ke_1+e_2+e_3},\\
            (\pi^K_{\mu^{(2)}(k)})^\ast\otimes\pi_\fq^\ast
            &=\pi^K_{(2-k)e_1}\oplus\pi^K_{(3-k)e_1+2e_2+e_3+\cdots+e_n}.
        \end{aligned}
    \end{equation}
    Every summand has multiplicity one. Among the highest weights displayed
    in~\eqref{eq:K-tensor-decompositions}, the only ones for which the
    $H$-trivial part is nonzero are
    \begin{equation}
        \bigl(\cH^K_{-(k+1)e_1}\bigr)_0\cong\C,
        \qquad
        \bigl(\cH^K_{(2-k)e_1}\bigr)_0\cong\C,
    \end{equation}
    arising from the first and fourth lines, respectively.  The $H$-trivial
    parts of all the other displayed summands vanish.
\end{lemma}
\begin{proof}
    The restriction of $\pi^K_{-e_1+e_2}$ to $H$ is the standard
    representation $\pi^H_{e_2}$. We therefore apply
    Proposition~\ref{thm:Ikeda} with $b_2=1$ and
    $b_3=\cdots=b_n=0$. The interlacing conditions~\eqref{eq:branching}
    become
    \[
        a_1\geq q+1\geq a_2\geq q\geq a_3\geq q\geq\cdots\geq a_n\geq q.
    \]
    Hence $a_3=\cdots=a_n=q$ and $a_2\in\{q,q+1\}$.
    If $a_2=q$, write $a_1=q+a$ with $a\geq1$ and put
    $q=a-k-1$. This gives $\lambda^{(1)}(a,k)$ with $k\leq a-1$.
    If $a_2=q+1$, the same parametrisation with $q=a-k$ gives
    $\lambda^{(2)}(a,k)$ with $k\leq a$.
    Conversely, each weight in these two families satisfies the above
    interlacing conditions. Hence~\eqref{eq:possible-lambdas} is exactly
    the list of highest weights for which
    $\Hom_H(\cH^G_\lambda,\cH^K_{-e_1+e_2})\neq0$.

    It remains to identify, for the representatives in
    \eqref{eq:possible-lambdas}, the $K$-components which can contribute.
    Applying the interlacing rule again gives
    \begin{equation*}
        \begin{aligned}
            \left.\pi^G_{\lambda^{(1)}(a,k)}\right|_K
            &=\bigoplus_{r=0}^{a}\bigoplus_{q=0}^{a-k-1}\pi^K_{\nu^{(1)}_{r,q}},\\
            \left.\pi^G_{\lambda^{(2)}(a,k)}\right|_K
            &=\bigoplus_{r=0}^{a-1}\bigoplus_{s=0}^{1}\bigoplus_{q=0}^{a-k}\pi^K_{\nu^{(2)}_{r,s,q}},
        \end{aligned}
    \end{equation*}
    where
    \begin{equation*}
        \begin{aligned}
            \nu^{(1)}_{r,q}
            &=(2a-k-1-r-2q)e_1+r e_2+(a-k-1-q)\sum_{j=2}^{n}e_j,\\
            \nu^{(2)}_{r,s,q}
            &=(2a-k-r-s-2q)e_1+(r+1)e_2+s e_3+(a-k-q)\sum_{j=2}^{n}e_j.
        \end{aligned}
    \end{equation*}
    Every summand occurs with multiplicity one.

    We retain only those components whose restrictions to $H$ occur in
    $\fq_\C$. In the first decomposition these are precisely
    \[
        \nu^{(1)}_{1,a-k-1}=\mu^{(1)}(k),
        \qquad
        \nu^{(1)}_{0,a-k-2}=\mu^{(2)}(k+3),
    \]
    where the second component occurs exactly when $k\leq a-2$. In the
    second decomposition,
    \[
        \nu^{(2)}_{0,0,a-k}=\mu^{(1)}(k)
    \]
    always occurs. If $n=3$, there is one further relevant component,
    \[
        \nu^{(2)}_{0,1,a-k}=\mu^{(2)}(k-1).
    \]
    If $n\geq4$, no further component of the form $\mu^{(2)}(\ell)$ occurs.
    Grouping all remaining irreducible summands into $\tau^{(j)}$ gives
    \eqref{eq:relevant-K-components}.

    Finally, let $V$ be the standard $n$-dimensional representation of $K$
    on the last $n$ coordinates and write $\chi_t=\pi^K_{te_1}$. Then
    \[
        \pi^K_{\mu^{(1)}(k)}=\chi_k\otimes V,
        \qquad
        \pi^K_{\mu^{(2)}(k)}=\chi_{k-1}\otimes V^\ast,
    \]
    while
    \[
        \pi_\fq=\chi_{-1}\otimes V,
        \qquad
        \pi_\fq^\ast=\chi_1\otimes V^\ast.
    \]
    Using
    \[
        V^\ast\otimes V\cong\C\oplus\mathfrak{sl}(V),
        \qquad
        V\otimes V\cong\operatorname{Sym}^2V\oplus\Lambda^2V,
    \]
    together with their duals, gives~\eqref{eq:K-tensor-decompositions}.
    Since $\chi_t$ restricts trivially to $H$, whereas
    $\mathfrak{sl}(V)$, $\operatorname{Sym}^2V$, $\Lambda^2V$ and their
    duals contain no trivial $H$-summand for $n\geq3$, we obtain
    \[
        \dim_\C\bigl(\cH^K_{-(k+1)e_1}\bigr)_0
        =\dim_\C\bigl(\cH^K_{(2-k)e_1}\bigr)_0=1,
    \]
    while the $H$-trivial parts of all the other components in
    \eqref{eq:K-tensor-decompositions} vanish.
\end{proof}
\begin{lemma}\label{lemma:Casimir-sun+1}
    Assume $n\geq3$, and let $\lambda^{(1)}(a,k)$ and $\lambda^{(2)}(a,k)$ be the weights in~\eqref{eq:possible-lambdas}, with the corresponding parameter ranges. Then
    \begin{small}
    \begin{equation*}
            \begin{aligned}
            c_G\bigl(\lambda^{(1)}(a,k)\bigr)
            &=
            \frac{2}{n+1}\Big[2(n+1)a(a-k+n-1)+nk(k-n+1)-n^2\Big],\\
            c_G\bigl(\lambda^{(2)}(a,k)\bigr)
            &=
            \frac{2}{n+1}\Big[2(n+1)a(a-k+n)+nk(k-n-1)-2k+n^2-2\Big].
            \end{aligned}
    \end{equation*}
    \end{small}
    For $\mu^{(1)}(k)=ke_1+e_2$ and $\mu^{(2)}(k)=ke_1+e_2+\cdots+e_n$, with any $k\in\Z$, we also have
    \begin{equation*}
            \begin{aligned}
            c_K\bigl(\mu^{(1)}(k)\bigr)&=\frac{2}{n+1}\left[nk^2-2k+n^2+n-1\right],\\
            c_K\bigl(\mu^{(2)}(k)\bigr)&=\frac{2}{n+1}\left[nk^2-2nk+2k+n^2+2n-3\right],\\
            c_K\bigl(-(k+1)e_1\bigr)&=\frac{2n}{n+1}(k+1)^2,\\
            c_K\bigl((2-k)e_1\bigr)&=\frac{2n}{n+1}(k-2)^2.
            \end{aligned}
    \end{equation*}
    In particular, the adjoint highest weight is $\mu_\fg=e_1-e_{n+1}=2e_1+e_2+\cdots+e_n$, and
    \begin{equation*}
        c_G(\mu_\fg)=4(n+1),\qquad
        c_K(\mu_\fq)=c_K\bigl(\mu^{(1)}(-1)\bigr)
        =c_K\bigl(\mu^{(2)}(2)\bigr)=2(n+1).
    \end{equation*}
    Thus $c_K(\mu_\fq)-c_G(\mu_\fg)=-2(n+1)$.
\end{lemma}
\begin{proof}
    Since $e_{n+1}=-\sum_{j=1}^n e_j$, the half-sums of positive roots satisfy
    \begin{equation*}
        \begin{aligned}
            2\rho_G &= \sum_{j=1}^n(2n+2-2j)e_j,\\
            2\rho_K&= \sum_{j=2}^n(2n+2-2j)e_j+(n-1)e_1,
        \end{aligned}
    \end{equation*}
    and the inner product induced by $\bangle{X,Y}=-\frac12\tr(XY)$ is
    \begin{equation*}
        \bangle{\sum_{j=1}^n x_je_j,\sum_{j=1}^n y_je_j}
        =2\sum_{j=1}^n x_jy_j-\frac{2}{n+1}
        \left(\sum_{j=1}^n x_j\right)\left(\sum_{j=1}^n y_j\right).
    \end{equation*}
    Consequently, for a highest weight $\eta=\sum_{j=1}^n x_je_j$, writing $S=\sum_{j=1}^n x_j$, the identity $c_L(\eta)=\bangle{\eta,\eta+2\rho_L}$ gives
    \begin{equation*}
        \begin{aligned}
            c_G(\eta)&=2\left[\sum_{j=1}^n x_j^2-\frac{S^2}{n+1}
            +\sum_{j=1}^n(n+2-2j)x_j\right],\\
            c_K(\eta)&=2\left[\sum_{j=1}^n x_j^2-\frac{S^2}{n+1}
            +\sum_{j=2}^n(n+3-2j)x_j\right].
        \end{aligned}
    \end{equation*}
    For $\lambda^{(1)}(a,k)$, the coefficient vector is $(2a-k-1,a-k-1,\ldots,a-k-1)$, so
    \begin{equation*}
        S=(n+1)a-nk-n,\qquad
        \sum_{j=1}^n x_j^2=(2a-k-1)^2+(n-1)(a-k-1)^2,
    \end{equation*}
    while $\sum_{j=1}^n(n+2-2j)x_j=n(2a-k-1)$. Substitution into the formula for $c_G$ yields the first identity. For $\lambda^{(2)}(a,k)$, the coefficient vector is $(2a-k,a-k+1,a-k,\ldots,a-k)$, giving
    \begin{equation*}
        S=(n+1)a-nk+1,\qquad
        \sum_{j=1}^n x_j^2=(2a-k)^2+(a-k+1)^2+(n-2)(a-k)^2,
    \end{equation*}
    and $\sum_{j=1}^n(n+2-2j)x_j=n(2a-k)+n-2$. This yields the second identity.

    For the two $K$-weights $\mu^{(1)}(k)$ and $\mu^{(2)}(k)$, the triples
    $\bigl(S,\sum_{j=1}^n x_j^2,\sum_{j=2}^n(n+3-2j)x_j\bigr)$ are, respectively,
    \begin{equation*}
        (k+1,k^2+1,n-1)\quad\text{and}\quad
        (k+n-1,k^2+n-1,n-1).
    \end{equation*}
    These give the stated $K$-Casimir eigenvalues. Since $c_K(te_1)=\frac{2n}{n+1}t^2$ for every $t\in\Z$, the two character formulas follow as well. Finally, substituting $\mu_\fg=2e_1+e_2+\cdots+e_n$, $\mu^{(1)}(-1)$ and $\mu^{(2)}(2)$ gives the last assertions.
\end{proof}
\begin{proposition}\label{prop:Hopf-negative-contributions}
    Assume $n\geq3$ and $\epsilon>0$, and set $\delta=1$. For the
    representatives $\lambda^{(1)}(a,k)$ and $\lambda^{(2)}(a,k)$ in
    Lemma~\ref{lemma:Hopf-branching}, the eigenvalues of
    $\cJ_\phi$ are given by the following four families:
    \begin{equation}\label{eq:Hopf-all-families-ranges}
        \begin{aligned}
            C^{(1)}_{a,k}(\epsilon)&\coloneqq C_{\lambda^{(1)}(a,k)}\bigl(\mu^{(1)}(k),-(k+1)e_1\bigr), &&a\geq1,\quad k\leq a-1,\\
            C^{(2)}_{a,k}(\epsilon)&\coloneqq C_{\lambda^{(2)}(a,k)}\bigl(\mu^{(1)}(k),-(k+1)e_1\bigr), &&a\geq1,\quad k\leq a,\\
            D^{(1)}_{a,k}(\epsilon)&\coloneqq C_{\lambda^{(1)}(a,k)}\bigl(\mu^{(2)}(k+3),-(k+1)e_1\bigr), &&a\geq1,\quad k\leq a-2,\\
            D^{(2)}_{a,k}(\epsilon)&\coloneqq C_{\lambda^{(2)}(a,k)}\bigl(\mu^{(2)}(k-1),(3-k)e_1\bigr), &&n=3,\quad a\geq1,\quad k\leq a.
        \end{aligned}
    \end{equation}
    Each displayed contribution has multiplicity
    $d_{\lambda^{(j)}(a,k)}=\dim_\C\cH^G_{\lambda^{(j)}(a,k)}$ in the
    chosen complex summand. By
    Remark~\ref{rem:real-vs-complex-multiplicity}, the weights in
    Lemma~\ref{lemma:Hopf-branching} give the complete real spectrum, with
    every such multiplicity doubled.

    More explicitly,
    \begin{equation}\label{eq:Hopf-family-formulas}
        \begin{aligned}
            C^{(1)}_{a,k}(\epsilon)
            &=2\Big[n(a-2)+2(a-1)+(a-k-1)(2a+n-2)\Big]
            +\frac{2n}{n+1}(k+1)^2\epsilon^2,\\
            C^{(2)}_{a,k}(\epsilon)
            &=2\Big[n(a-1)-2+(a-k)(2a+n)\Big]
            +\frac{2n}{n+1}(k+1)^2\epsilon^2,\\
            D^{(1)}_{a,k}(\epsilon)
            &=2\Big[n(a-1)+(a-k-2)(2a+n+2)\Big]
            +\frac{2n}{n+1}(k+1)^2\epsilon^2.
        \end{aligned}
    \end{equation}
    If $n=3$, then
    \begin{equation}\label{eq:D2-formula}
        D^{(2)}_{a,k}(\epsilon)
        =2\Big[5a-7+(a-k)(2a+1)\Big]+\frac32(k-3)^2\epsilon^2.
    \end{equation}

    Consequently,
    \begin{equation*}
        \begin{aligned}
            C^{(1)}_{a,k}(\epsilon)<0
            &\quad\Longleftrightarrow\quad (a,k)=(1,0)
            \text{ and }\epsilon^2<n+1,\\
            C^{(2)}_{a,k}(\epsilon)<0
            &\quad\Longleftrightarrow\quad (a,k)=(1,1)
            \text{ and }\epsilon^2<\frac{n+1}{2n},\\
            D^{(2)}_{a,k}(\epsilon)<0\quad(n=3)
            &\quad\Longleftrightarrow\quad (a,k)=(1,1)
            \text{ and }\epsilon^2<\frac23.
        \end{aligned}
    \end{equation*}
    All $D^{(1)}_{a,k}(\epsilon)$ are non-negative. Moreover,
    \begin{equation}\label{eq:permanent-zero}
        C^{(1)}_{1,-1}(\epsilon)=D^{(1)}_{1,-1}(\epsilon)=0
        \qquad\text{for every }\epsilon>0.
    \end{equation}
    Apart from~\eqref{eq:permanent-zero}, zero occurs only when
    \begin{equation*}
        C^{(1)}_{1,0}(\epsilon)=0\iff\epsilon^2=n+1
    \end{equation*}
    or
    \begin{equation*}
        C^{(2)}_{1,1}(\epsilon)=0\iff\epsilon^2=\frac{n+1}{2n}.
    \end{equation*}
    If $n=3$ and $\epsilon^2=\frac{n+1}{2n}=\frac23$, then one also has
    $D^{(2)}_{1,1}(\epsilon)=0$.
\end{proposition}

\begin{proof}
    Theorem~\ref{thm:J-eigenvalues}, together with
    Lemma~\ref{lemma:Hopf-branching}, gives the four contributions in
    \eqref{eq:Hopf-all-families-ranges}. A $K$-component
    $\pi^K_{\mu^{(1)}(k)}$ contributes through the scalar summand
    $-(k+1)e_1$ in $(\pi^K_{\mu^{(1)}(k)})^\ast\otimes\pi_\fq$, while
    $\pi^K_{\mu^{(2)}(\ell)}$ contributes through $(2-\ell)e_1$ in
    $(\pi^K_{\mu^{(2)}(\ell)})^\ast\otimes\pi_\fq^\ast$; taking
    $\ell=k+3$ and $\ell=k-1$ gives the last two families. The relevant
    $H$-fixed spaces are one-dimensional, so each contribution has complex
    multiplicity $d_{\lambda^{(j)}(a,k)}$, and
    Remark~\ref{rem:real-vs-complex-multiplicity} gives the full real
    spectrum by doubling these multiplicities. Substitution of the
    Casimir values from Lemma~\ref{lemma:Casimir-sun+1} yields
    \eqref{eq:Hopf-family-formulas} and~\eqref{eq:D2-formula}.

    For $C^{(1)}_{a,k}$, the condition $a-k-1\geq0$ and
    $n(a-2)+2(a-1)>0$ show positivity when $a\geq2$. If $a=1$, then
    $k\leq0$ and $C^{(1)}_{1,k}(\epsilon)=-2n(k+1)+\frac{2n}{n+1}(k+1)^2\epsilon^2$;
    this is negative only for $k=0$ and $\epsilon^2<n+1$, vanishes
    identically for $k=-1$, and is positive for $k\leq-2$. Similarly,
    $C^{(2)}_{a,k}>0$ for $a\geq2$, since $n(a-1)-2\geq n-2>0$; for
    $a=1$, its constant term is $2[-2+(1-k)(n+2)]$, which is positive for
    $k\leq0$, while $k=1$ gives
    $C^{(2)}_{1,1}(\epsilon)=\frac{8n}{n+1}\epsilon^2-4$, negative exactly
    when $\epsilon^2<(n+1)/(2n)$.

    Finally, all terms of $D^{(1)}_{a,k}$ are non-negative because
    $k\leq a-2$, and equality holds only for $(a,k)=(1,-1)$, proving
    \eqref{eq:permanent-zero}. If $n=3$, the constant term of
    $D^{(2)}_{a,k}$ is positive for $a\geq2$, and for $a=1$ it equals
    $2(1-3k)$, which is positive for $k\leq0$; the remaining case gives
    $D^{(2)}_{1,1}(\epsilon)=6\epsilon^2-4=C^{(2)}_{1,1}(\epsilon)$.
    The stated zero cases follow from the same formulas.
\end{proof}
\begin{theorem}\label{thm:Hopf-negative-multiplicities}
    Let $n\geq3$ and $\epsilon>0$. For the complex Hopf fibration
    \[
        \phi:(\bbS^{2n+1},g_\epsilon)\longrightarrow(\C P^n,h),
    \]
    the only eigenvalues of $\cJ_\phi$ which can be negative are
    $\frac{2n}{n+1}\epsilon^2-2n$ and
    $\frac{8n}{n+1}\epsilon^2-4$, with real multiplicities $2(n+1)$ and
    $n(n+1)$, respectively. The Morse index and nullity of $\phi$ are
    \begin{align}
        \mathrm{Ind}(\phi)&=
        \begin{cases}
            (n+1)(n+2),&\epsilon^2\in\left(0,\dfrac{n+1}{2n}\right),\\[1ex]
            2(n+1),&\epsilon^2\in\left[\dfrac{n+1}{2n},n+1\right),\\[1ex]
            0,&\epsilon^2\in[n+1,\infty).
        \end{cases}\label{eq:Hopf-Morse-index}\\
        \mathrm{Null}(\phi)&=
        \begin{cases}
            n(3n+5),&\epsilon^2=\dfrac{n+1}{2n},\\[1ex]
            2(n^2+3n+1),&\epsilon^2=n+1,\\[1ex]
            2n(n+2),&\text{otherwise}.
        \end{cases}\label{eq:Hopf-nullity}
    \end{align}

    In particular, for the unit round metric
    $\epsilon^2=(n+1)/(2n)$, the only negative eigenvalue is $1-2n$, with
    real multiplicity $2n+2$, and
    \begin{equation}
        \mathrm{Ind}(\phi)=2n+2,
        \qquad
        \mathrm{Null}(\phi)=n(3n+5).
    \end{equation}
\end{theorem}
\begin{proof}
    By Lemma~\ref{lemma:Hopf-branching} and
    Remark~\ref{rem:real-vs-complex-multiplicity}, the weights listed in
    the lemma give all the eigenvalues in
    \eqref{eq:Hopf-all-families-ranges}, with doubled multiplicities.
    Proposition~\ref{prop:Hopf-negative-contributions} shows that the only
    possibly negative ones are
    $C^{(1)}_{1,0}(\epsilon)=\frac{2n}{n+1}\epsilon^2-2n$ and
    $C^{(2)}_{1,1}(\epsilon)=\frac{8n}{n+1}\epsilon^2-4$.
    They correspond, respectively, to
    $\lambda^{(1)}(1,0)=e_1$ and
    $\lambda^{(2)}(1,1)=e_1+e_2$, whose representation dimensions are
    $n+1$ and $\binom{n+1}{2}$. Hence
    their real multiplicities are $2(n+1)$ and
    $2\binom{n+1}{2}=n(n+1)$; moreover, the first is negative precisely
    when $\epsilon^2<n+1$, while the second is negative precisely when
    $\epsilon^2<(n+1)/(2n)$.

    The permanent zero eigenvalue corresponds to
    $\lambda^{(1)}(1,-1)=\mu_\fg$, the adjoint representation, of
    dimension $n(n+2)$, and therefore contributes $2n(n+2)$ to the
    nullity. At $\epsilon^2=(n+1)/(2n)$ one must add $n(n+1)$, obtaining
    $n(3n+5)$, whereas at $\epsilon^2=n+1$ one adds $2(n+1)$, obtaining
    $2(n^2+3n+1)$. Proposition~\ref{prop:Hopf-negative-contributions}
    excludes any further zero eigenvalues, proving
    \eqref{eq:Hopf-nullity}.

    Finally, both eigenvalues are negative when
    $0<\epsilon^2<(n+1)/(2n)$, giving index
    $2(n+1)+n(n+1)=(n+1)(n+2)$; only the first eigenvalue is negative when
    $(n+1)/(2n)\leq\epsilon^2<n+1$, giving index $2(n+1)$; and neither is
    negative when $\epsilon^2\geq n+1$. This proves
    \eqref{eq:Hopf-Morse-index}. For the round metric
    $\epsilon^2=(n+1)/(2n)$, these eigenvalues are $1-2n$ and $0$,
    respectively, which gives the final assertion.
\end{proof}
\section{Harmonic maps with low Morse index}
\subsection{The case of \texorpdfstring{$\mathbb{S}^3\to\mathbb{S}^2$}{} revisited}
We use the standard Euclidean models
\begin{equation*}
    \bbS^3=\left\{z=(z_1,z_2)\in\C^2\;:\;|z_1|^2+|z_2|^2=1\right\}
\end{equation*}
and
\begin{equation*}
    \bbS^2=\left\{u=(w,t)\in\C\times\R\;:\;|w|^2+t^2=1\right\}.
\end{equation*}
The ambient spaces are endowed with their real Euclidean inner products. If $u=(w,t)\in\bbS^2$ and $v=(\eta,s)\in T_u\bbS^2$, the complex structure compatible with the standard orientation is
\begin{equation}\label{eq:complex-structure-S2}
    J_uv=u\times v
    =\bigl(i(t\eta-sw),\operatorname{Im}(\bar w\eta)\bigr),
\end{equation}
where we identify $\C\times\R$ with $\R^3$. In particular, $J$ is parallel and orthogonal.

For each $a\in\C^2\cong\R^4$, let $X_a\in\fX(\bbS^3)$ be the associated conformal Killing vector field, defined by
\begin{equation}\label{eq:conformal-fields-C2}
    X_a(z)=a-\bangle{a,z}z.
\end{equation}
Fix the standard real orthonormal basis $\{a_\alpha\}_{\alpha=1}^4$ of $\C^2$, given by
\begin{equation*}
    a_1=(1,0),\qquad a_2=(i,0),\qquad
    a_3=(0,1),\qquad a_4=(0,i).
\end{equation*}
The corresponding vector fields satisfy $\nabla_YX_{a_\alpha}=-\bangle{a_\alpha,z}Y$ and the completeness relation
\begin{equation}\label{eq:conformal-fields-completeness}
    \sum_{\alpha=1}^4\bangle{Y,X_{a_\alpha}}X_{a_\alpha}=Y
    \qforq Y\in T_z\bbS^3.
\end{equation}
The Hopf fibration in these coordinates is
\begin{equation}\label{eq:Hopf-map-C2}
    \pi(z_1,z_2)
    =\bigl(2z_1\bar z_2,|z_1|^2-|z_2|^2\bigr)
    \in\bbS^2\subset\C\times\R.
\end{equation}

We now state and prove the rigidity result of Rivière in a form adapted to the notation of this paper. This rewriting is the key idea for proving our main result in the next section. 

\begin{theorem}[Rivière~\cite{Riviere2023-co}]\label{thm:Riviere-rigidity}
    Let $\phi:\bbS^3\to\bbS^2$ be a non-constant smooth harmonic map. If $\mathrm{Ind}(\phi)\leq4$, then $\mathrm{Ind}(\phi)=4$ and $\phi$ is a harmonic morphism. More precisely, there exist an isometry $f\in\Aut(\bbS^3)$ and a holomorphic map $P:\C P^1\to\C P^1$ such that
    \begin{equation}
        \phi=P\circ\pi\circ f,
    \end{equation}
    where $\pi$ is the Hopf fibration defined in~\eqref{eq:Hopf-map-C2}.
\end{theorem}

\begin{proof}
    Let $E=\phi^{-1}(T\bbS^2)$ and denote by $\nabla^\phi$ its induced connection. For a section $v\in\Gamma(E)$, the Jacobi operator is
    \begin{equation}\label{eq:Jacobi-S3-S2}
        \cJ_\phi v
        =(\nabla^\phi)^\ast\nabla^\phi v
        -\sum_{i=1}^3R^h(v,\phi_i)\phi_i,
        \qquad \phi_i=d\phi(e_i),
    \end{equation}
    where $\{e_i\}_{i=1}^3$ is a local orthonormal frame on $\bbS^3$. The standard variational identity associated with conformal transformations of the domain gives
    \begin{equation}\label{eq:conformal-negative-directions}
        \cJ_\phi(v_a)=-v_a,
        \qquad v_a\coloneqq d\phi(X_a),
        \qquad a\in\C^2.
    \end{equation}
    In particular,
    \begin{equation*}
        \cJ_\phi(v_{a_\alpha})=-v_{a_\alpha},
        \qquad \alpha=1,\ldots,4.
    \end{equation*}
    For a non-constant smooth harmonic map, these four eigensections are linearly independent. Indeed, if $v_a=0$ for some $a\neq0$, then, after composing with an isometry of $\bbS^3$, $\phi$ is constant along the conformal flow generated by $X_{a_1}$. Letting this flow converge to its attracting fixed point shows that $\phi$ is constant, a contradiction. Thus $\mathrm{Ind}(\phi)\geq4$. Under the hypothesis of the theorem, it follows that $\mathrm{Ind}(\phi)=4$ and $-1$ is the least eigenvalue of $\cJ_\phi$.

    Since $J$ is parallel, equations~\eqref{eq:Jacobi-S3-S2} and~\eqref{eq:conformal-negative-directions} imply
    \begin{equation}\label{eq:Jacobi-on-JV-S2}
        \cJ_\phi(Jv_a)
        =-Jv_a+
        \sum_{i=1}^3\left\{
        R^h(v_a,\phi_i)J\phi_i
        -R^h(Jv_a,\phi_i)\phi_i
        \right\}.
    \end{equation}
    The curvature tensor of the unit sphere satisfies
    \begin{equation*}
        R^h(U,V)Z=h(V,Z)U-h(U,Z)V.
    \end{equation*}
    Consequently,
    \begin{equation*}
        \begin{aligned}
        h\bigl(R^h(v_a,\phi_i)J\phi_i,Jv_a\bigr)
        &=h(v_a,J\phi_i)^2,\\
        h\bigl(R^h(Jv_a,\phi_i)\phi_i,Jv_a\bigr)
        &=|\phi_i|^2|v_a|^2-h(v_a,J\phi_i)^2.
        \end{aligned}
    \end{equation*}
    Pairing~\eqref{eq:Jacobi-on-JV-S2} with $Jv_a$ gives, pointwise,
    \begin{equation*}
        h(\cJ_\phi(Jv_a),Jv_a)
        =-|v_a|^2-|d\phi|^2|v_a|^2
        +2\sum_{i=1}^3h(v_a,J\phi_i)^2.
    \end{equation*}
    To take the trace in the parameter $a$, fix $z\in\bbS^3$ and choose
    an orthonormal basis $\{e_1,e_2,e_3\}$ of $T_z\bbS^3$. Completing it
    with the unit normal $z$ gives an orthonormal basis
    $\{e_1,e_2,e_3,z\}$ of $\R^4$. The left-hand side, regarded as a
    quadratic form in $a\in\R^4$, has the same trace in this adapted basis
    as in the fixed basis $\{a_\alpha\}_{\alpha=1}^4$. Here, for the fixed
    point $z$, the vectors $e_i$ are regarded as constant elements of
    $\R^4$ defining the global fields $X_{e_i}$; no differentiation of a
    moving frame is involved. Since
    \[
        X_{e_j}(z)=e_j,\qquad X_z(z)=0,
    \]
    we have $v_{e_j}(z)=\phi_j$ and $v_z(z)=0$. Therefore, tracing the
    preceding identity gives
    \begin{equation*}
        \begin{aligned}
        \sum_{\alpha=1}^4h(\cJ_\phi(Jv_{a_\alpha}),Jv_{a_\alpha})
        &=-\sum_{j=1}^3|\phi_j|^2
        -\sum_{i,j=1}^3|\phi_i|^2|\phi_j|^2
        +2\sum_{i,j=1}^3h(\phi_j,J\phi_i)^2\\
        &=-|d\phi|^2-|d\phi|^4
        +2\left|d\phi^\ast\circ J\circ d\phi\right|^2.
        \end{aligned}
    \end{equation*}
    Thus
    \begin{equation}\label{eq:Riviere-pointwise-identity}
        \begin{aligned}
        \sum_{\alpha=1}^4h(\cJ_\phi(Jv_{a_\alpha}),Jv_{a_\alpha})
        &=-|d\phi|^2-|d\phi|^4
        +2\left|d\phi^\ast\circ J\circ d\phi\right|^2\\
        &=-|d\phi|^2+|d\phi|^4-2|\phi^\ast h|^2.
        \end{aligned}
    \end{equation}
    Here the pullback metric $\phi^\ast h$ is identified with the self-adjoint endomorphism $d\phi^\ast\circ d\phi$, and the second equality follows from the fact that the target has real dimension two.

    Since $\lambda_1(\cJ_\phi)=-1$, the Rayleigh inequality gives
    \begin{equation*}
        \int_{\bbS^3}\left(h(\cJ_\phi v,v)+|v|^2\right)\,\mathrm{d}v_g\geq0
        \qforq v\in\Gamma(E).
    \end{equation*}
    Applying this inequality to the four sections $Jv_{a_\alpha}$, summing, and observing that
    \begin{equation*}
        \sum_{\alpha=1}^4|Jv_{a_\alpha}|^2
        =\sum_{\alpha=1}^4|d\phi(X_{a_\alpha})|^2
        =|d\phi|^2,
    \end{equation*}
    equation~\eqref{eq:Riviere-pointwise-identity} yields
    \begin{equation}\label{ineq:Riviere}
        \int_{\bbS^3}\left(|d\phi|^4-2|\phi^\ast h|^2\right)\,\mathrm{d}v_g\geq0.
    \end{equation}

    Let $\sigma_1\geq\sigma_2\geq0$ be the two possibly non-zero singular values of $d\phi$ at a given point. Then
    \begin{equation}\label{eq:Riviere-algebraic-inequality}
        |d\phi|^4-2|\phi^\ast h|^2
        =(\sigma_1^2+\sigma_2^2)^2
        -2(\sigma_1^4+\sigma_2^4)
        =-(\sigma_1^2-\sigma_2^2)^2\leq0.
    \end{equation}
    Combining~\eqref{ineq:Riviere} and~\eqref{eq:Riviere-algebraic-inequality}, we conclude that $\sigma_1=\sigma_2$ at every point. Hence $\phi$ is horizontally weakly conformal. Since it is also harmonic, it is a harmonic morphism. The classification of smooth harmonic morphisms $\bbS^3\to\bbS^2$ in~\cite[Theorem~6.7.7]{BairdWood2003} then gives
    \begin{equation*}
        \phi=P\circ\pi\circ f,
    \end{equation*}
    for an isometry $f\in\Aut(\bbS^3)$ and a holomorphic self-map $P$ of $\C P^1$.
\end{proof}
\subsection{The higher dimensional case}
In this subsection, let $\phi:\bbS^{2n+1}\to\C P^n$ be a harmonic map and set $E=\phi^{-1}(T\C P^n)$. We show that, if $\mathrm{Ind}(\phi)=2n+2$, then $\mathrm{Im}(d\phi_x)$ is a complex subspace of $T_{\phi(x)}\C P^n$ for every $x\in\bbS^{2n+1}$. The argument follows the same strategy as above, using the explicit curvature tensor of $\C P^n$.

\begin{lemma}\label{lemma:Jacobian vs J}
    Let $v\in\Gamma(E)$ and let $J$ denote the complex structure on $\C P^n$. Then
    \begin{equation}
    h(\cJ_\phi(Jv),Jv) = h(\cJ_\phi v,v) + Z(v),
    \end{equation}
    where
    \begin{equation}
        Z(v)\coloneqq \sum_{j=1}^{2n+1}\left(h(v,J\phi_j)^2-h(v,\phi_j)^2\right),
        \qquad \phi_j=d\phi(e_j),
    \end{equation}
    for any local orthonormal frame $\{e_j\}_{j=1}^{2n+1}$ on $\bbS^{2n+1}$.
\end{lemma}
\begin{proof}
    Since $\C P^n$ is Kähler, $J$ is parallel and commutes with the rough Laplacian. Therefore,
    \begin{equation}\label{eq:J_phi composed with J}
        \cJ_\phi(Jv) = J(\cJ_\phi v) + \sum_{i=1}^{2n+1}JR^h(v,\phi_i)\phi_i - \sum_{i=1}^{2n+1}R^h(Jv,\phi_i)\phi_i,
    \end{equation}
    and pairing this with $Jv$ we obtain
    \begin{equation}\label{eq:J_phi paired}
        h(\cJ_\phi(Jv),Jv) = h(\cJ_\phi v,v) + \sum_{i=1}^{2n+1}h(R^h(v,\phi_i)\phi_i,v) - \sum_{i=1}^{2n+1}h(R^h(Jv,\phi_i)\phi_i,Jv).
    \end{equation}
    For $X,Y,Z\in T_q\C P^n$, the curvature tensor is
    \begin{equation}\label{eq:CPn curvature}
        R^h(X,Y)Z=\frac14\bigl(h(Y,Z)X-h(X,Z)Y+h(JY,Z)JX-h(JX,Z)JY+2h(X,JY)JZ\bigr).
    \end{equation}
    In particular,
    \begin{align*}
        h(R^h(X,Y)Y,X)&=\frac14\left(|X|^2|Y|^2-h(X,Y)^2+3h(X,JY)^2\right),\\
        h(R^h(JX,Y)Y,JX)&=\frac14\left(|X|^2|Y|^2-h(X,JY)^2+3h(X,Y)^2\right).
    \end{align*}
    Taking their difference in~\eqref{eq:J_phi paired} proves the claim.
\end{proof}
\begin{lemma}\label{lemm:ineq-with-pullbacks}
    Let $f:\R^n\to\R^{2m}\cong\C^m$ be a real linear map, let $h$ be the standard real inner product on $\C^m$, and let $\omega(X,Y)=h(X,JY)$ be its Kähler form, where $J = \begin{pmatrix}
        0&-\mathrm{I}\\ \mathrm{I}&0
    \end{pmatrix}$. Then
    \begin{equation}
        |f^\ast\omega|^2\leq |f^\ast h|^2,
    \end{equation}
    with equality if and only if $f\circ f^\ast:\R^{2m}\to\R^{2m}$ commutes with $J$, where $f^\ast$ denotes the adjoint of $f$.
\end{lemma}
\begin{proof}
    Fix orthonormal bases $\{v_i\}$ and $\{w_j\}$ of $\R^n$ and $\C^m$, respectively, write $fv_i=f_i^jw_j$, and let $F=(f_i^j)$. Then
    \begin{equation}
        |f^\ast h|^2 = \sum_{i,j=1}^n h(fv_i,fv_j)^2 = \sum_{i,j} \left(\sum_{k=1}f_i^kf_j^k\right)^2 = |F^{\mathsf t}F|^2,
    \end{equation}
    For a matrix $A$, we have $|A|^2=\tr(A^{\mathsf t}A)$, and hence $|f^\ast h|^2=\tr((F^{\mathsf t}F)^2)$. Similarly, $|f^\ast\omega|^2=|F^{\mathsf t}JF|^2$. Since $F^{\mathsf t}JF$ is skew-symmetric, setting $P=FF^{\mathsf t}$ gives
    \begin{equation*}
    \begin{aligned}
        |f^\ast\omega|^2 = -\tr((F^{\mathsf t}JF)^2) = -\tr(F^{\mathsf t}JFF^{\mathsf t}JF) = -\tr(JPJP).
    \end{aligned}
    \end{equation*}
    Therefore
    \begin{equation}
        |f^\ast h|^2 - |f^\ast \omega|^2   = \tr(P^2+JPJP).
    \end{equation}
    Now, writing as block matrices, we have
    \begin{equation}
        P = \begin{pmatrix}
            A&B\\B^{\mathsf t}&C
        \end{pmatrix} \qandq J = \begin{pmatrix}
        0&-\mathrm{I}\\
        \mathrm{I}&0
    \end{pmatrix},
    \end{equation}
    where $A$ and $C$ are symmetric matrices. Hence
    \begin{equation}
        \tr(JPJP) = \tr\left[\begin{pmatrix}
            -C&B^{\mathsf t}\\B&-A
        \end{pmatrix}\begin{pmatrix}
            A&B\\B^{\mathsf t}&C
            \end{pmatrix}\right] = 2\tr(B^2) - 2\tr(AC),
    \end{equation}
    and also
    \begin{equation}
        \tr(P^2) = \tr(A^2)+\tr(C^2)+2\tr(B^{\mathsf t}B).
    \end{equation}
    Putting all the information together we get
    \begin{equation}\label{eq:p2-jpjp}
        \tr(P^2+JPJP) = \tr((A-C)^2) + 2\tr(B^{\mathsf t}B+B^2).
    \end{equation}
    Since $A-C$ is a symmetric matrix, the first term is clearly non-negative. For the second term, we write $B=S+K$, where $S$ is symmetric and $K$ is skew-symmetric, and then
    \begin{equation}
        B^{\mathsf t}B+B^2=(S^2-K^2+SK-KS)+(S^2+K^2+SK+KS)=2S^2+2SK.
    \end{equation}
    Since $S$ and $K$ are orthogonal, $\tr(SK)=0$, so $\tr(B^{\mathsf t}B+B^2)=2\tr(S^2)\geq0$. This proves the inequality.

    Equality holds if and only if $A=C$ and $S=0$, that is, if and only if $B$ is skew-symmetric. In that case
    \begin{equation}
        P = \begin{pmatrix}
            A &B\\-B& A
        \end{pmatrix},
    \end{equation}
    which is equivalent to $PJ=JP$, or, intrinsically, to $(f\circ f^\ast)J=J(f\circ f^\ast)$.
\end{proof}
\begin{corollary}
    Given any orthonormal frame $\{e_i\}$ at a point $p\in\bbS^{2n+1}$, we have
    \begin{equation}
        \sum_{k=1}^{2n+1} Z(d\phi(e_k)) \leq 0.
    \end{equation}
    Equality holds if and only if $d\phi_p\circ d\phi_p^\ast$ commutes with
    $J_{\phi(p)}$.
\end{corollary}
\begin{proof}[Proof of Theorem~\ref{thm:intro-low-index-rigidity}]
For each $a\in\R^{2n+2}$, let $X_a\in\fX(\bbS^{2n+1})$ be the
conformal Killing vector field
\[
    X_a(x)=a-\bangle{a,x}x,
\]
and set $v_a\coloneqq d\phi(X_a)\in\Gamma(E)$. Since $\phi$ is harmonic,
\begin{equation}
    \cJ_\phi(v_a)=(1-2n)v_a,
\end{equation}
and Lemma~\ref{lemma:Jacobian vs J} gives $h(\cJ_\phi(Jv_a),Jv_a)=(1-2n)|v_a|^2+Z(v_a)$. Since $\dim\{v_a:a\in\R^{2n+2}\}=2n+2$ and $\mathrm{Ind}(\phi)=2n+2$, we have $\lambda_1(\cJ_\phi)=1-2n$. The Rayleigh quotient therefore yields
\begin{equation}\label{eq:Z-is-positive}
    0\leq\int_{\bbS^{2n+1}}\left[h(\cJ_\phi(Jv_a),Jv_a)-(1-2n)|v_a|^2\right]\,\mathrm{d}v_g
    =\int_{\bbS^{2n+1}}Z(v_a)\,\mathrm{d}v_g.
\end{equation}
For each $x\in\bbS^{2n+1}$, define a quadratic form on the fixed ambient
space $\R^{2n+2}$ by
\[
    Q_x(a)\coloneqq Z(v_a)(x).
\]
Equation~\eqref{eq:Z-is-positive} says that the averaged quadratic form
$Q=\int_{\bbS^{2n+1}}Q_x\,\mathrm{d}v_g$ is positive semidefinite. Hence
\begin{equation}\label{eq:integrated-trace-Q}
    0\leq\tr Q
    =\int_{\bbS^{2n+1}}\tr Q_x\,\mathrm{d}v_g,
\end{equation}
where the trace may first be taken with respect to any fixed orthonormal
basis of $\R^{2n+2}$.

We now compute $\tr Q_x$ pointwise using a basis adapted to $x$. Choose
an orthonormal basis of $T_x\bbS^{2n+1}$ and complete it with the unit
normal $x$, thus obtaining the orthonormal basis
\[
    \{e_1,\ldots,e_{2n+1},x\}
\]
of $\R^{2n+2}$. Since $X_{e_i}(x)=e_i$ and $X_x(x)=0$, the invariance of
the trace under a change of orthonormal basis gives
\begin{equation}
    \tr Q_x
    =\sum_{i=1}^{2n+1}Z(d\phi_x(e_i))
    =|\phi^\ast\omega|^2_x-|\phi^\ast h|^2_x\leq0.
\end{equation}
Combining this pointwise inequality with~\eqref{eq:integrated-trace-Q}
gives $\int_{\bbS^{2n+1}}\tr Q_x\,\mathrm{d}v_g=0$. Since
$x\mapsto\tr Q_x$ is continuous and non-positive, it follows that
$\tr Q_x=0$ for every $x$. Lemma~\ref{lemm:ineq-with-pullbacks} then implies
\begin{equation}\label{eq:PHWC-commutation}
    d\phi_x\circ d\phi_x^\ast\circ J_{\phi(x)}=J_{\phi(x)}\circ d\phi_x\circ d\phi_x^\ast
    \qforq x\in\bbS^{2n+1}.
\end{equation}
The next step is to prove that, after composing with an isometry of $\bbS^{2n+1}$, one has $d\phi(\xi)=0$, where $\xi$ is the Hopf vector field. Equations~\eqref{eq:J_phi composed with J} and~\eqref{eq:CPn curvature} give
 \begin{equation}\label{eq:J_phi and J again}
     \cJ_\phi(Jv)=J(\cJ_\phi v)-\sum_{i=1}^{2n+1}\left[h(Jv,\phi_i)J\phi_i-h(v,\phi_i)\phi_i\right].
 \end{equation}
Set $B(v,w)=h(Jv,w)Jw+h(v,w)w$ and note that $B(v,w)+B(v,Jw)=0$. Since $d\phi\circ d\phi^\ast$ commutes with $J$, there is an endomorphism $J_0$ of $T\bbS^{2n+1}$ satisfying $d\phi(J_0X)=Jd\phi(X)$ on $\mathrm{Im}(d\phi^\ast)$ and $J_0=0$ on $\ker(d\phi)$. If $\{e_1,\ldots,e_{2r}\}$ is an orthonormal basis of $\mathrm{Im}(d\phi^\ast)$ with $e_{r+i}=J_0e_i$, then
 \begin{equation}
    B(v,d\phi(e_i))+B(v,d\phi(e_{r+i}))=0,\qquad 1\leq i\leq r.
 \end{equation}
Thus~\eqref{eq:J_phi and J again} implies $J\circ\cJ_\phi=\cJ_\phi\circ J$. In particular, the eigenspace $\mathcal V=\{v_a:a\in\R^{2n+2}\}$ is $J$-invariant, so there is a complex structure $J_1$ on $\R^{2n+2}$ such that
 \begin{equation}\label{eq:v commutes with J}
     J v_a = v_{J_1a} \qforq a\in\R^{2n+2}.
 \end{equation}
Let $J$ also denote the standard complex structure of $\R^{2n+2}$, which induces the Sasakian structure on $\bbS^{2n+1}$ and satisfies $JX=-\nabla_X\xi$. Choose $M\in\O(2n+2)$ such that $J=M^\ast J_1M$. Since $X_{Ma}(Mx)=MX_a(x)$, we also have
 \begin{equation}
     M(X_{Ja}(x)) = J_1(Ma) - \bangle{J_1Ma,Mx}Mx = X_{J_1(Ma)}(Mx)
 \end{equation}
and, regarding $M$ as an isometry of $\bbS^{2n+1}$,
 \begin{equation}
     d(\phi\circ M)_x(X_{Ja})=d\phi_{Mx}(X_{J_1(Ma)}(Mx))=Jd\phi_{Mx}(X_{Ma}(Mx))=Jd(\phi\circ M)_x(X_a).
 \end{equation}
Replacing $\phi$ by $\phi\circ M$, equation~\eqref{eq:v commutes with J} becomes
\begin{equation}\label{eq:v commutes with J again}
    Jv_a = v_{Ja} \qforq a\in\R^{2n+2}.
\end{equation}
Since $\xi(x)=Jx=X_{Jx}(x)$, equation~\eqref{eq:v commutes with J again} gives
\begin{equation}
    0=v_{-x}(x)=Jv_{Jx}(x)=Jd\phi_x(\xi(x)),
\end{equation}
so $d\phi(\xi)=0$. Hence $\phi$ is constant along the fibers of the Hopf fibration and factors as $\phi=P\circ\pi$ for a smooth map $P:\C P^n\to\C P^n$.

It remains to prove that $P$ is holomorphic. If $a\in T_x\bbS^{2n+1}$ is horizontal, then $Ja\in T_x\bbS^{2n+1}$ and
\begin{equation}
    Jd\phi_x(a) = Jd\phi_x(X_a(x)) = d\phi_x(X_{Ja}(x)) = d\phi_x(Ja),
\end{equation}
so $d\phi_x$ intertwines the complex structures on the horizontal distribution and on $\C P^n$. Thus, for $z=\pi(x)$ and $Y=d\pi_x(X)\in T_z\C P^n$,
\begin{equation}
    dP_z(JY)=d\phi_x(JX)=Jd\phi_x(X)=JdP_z(Y),
\end{equation}
which proves that $P$ is holomorphic.
\end{proof}

\section*{Acknowledgements}
The author extends special thanks to Professor Eric Loubeau.

This study was financed in part by the Coordenação de Aperfeiçoamento de
Pessoal de Nível Superior -- Brasil (CAPES).

The author also acknowledges support from the CAPES/MATH-AMSUD programme
under grant no.~88887.193117/2025-00.

The author thanks the Laboratoire de Mathématiques de Bretagne Atlantique
(LMBA, UMR~6205), Université de Bretagne Occidentale (UBO), for its
hospitality while part of this research was carried out.

\section*{Statement on the use of AI-assisted tools}
Prism was used for editorial purposes, including \LaTeX{} editing and
language revision.  OpenAI's ChatGPT, using the GPT-5.6 Sol model, was
used as an assistant for the representation-theoretic and arithmetic
computations in Section~3.  All mathematical arguments and computations
were independently verified by the author, who takes full responsibility
for the content of this article.

\bibliographystyle{alpha}
\bibliography{bib}

@article{EellsSampson1964,
  author  = {Eells, Jr., James and Sampson, J. H.},
  title   = {Harmonic mappings of {Riemannian} manifolds},
  journal = {Amer. J. Math.},
  volume  = {86},
  number  = {1},
  pages   = {109--160},
  year    = {1964},
  doi     = {10.2307/2373037}
}

@book{BairdWood2003,
  author    = {Baird, Paul and Wood, John C.},
  title     = {Harmonic Morphisms between {Riemannian} Manifolds},
  series    = {London Mathematical Society Monographs. New Series},
  volume    = {29},
  publisher = {Oxford University Press},
  address   = {Oxford},
  year      = {2003},
  doi       = {10.1093/acprof:oso/9780198503620.001.0001}
}

@article{ElSoufi1995,
  author  = {El Soufi, Ahmad},
  title   = {Indice de {Morse} des applications harmoniques de la sph{\`e}re},
  journal = {Compositio Math.},
  volume  = {95},
  number  = {3},
  pages   = {343--362},
  year    = {1995}
}

@article{LoubeauOniciuc2007,
  author  = {Loubeau, Eric and Oniciuc, Cezar},
  title   = {On the biharmonic and harmonic indices of the {Hopf} map},
  journal = {Trans. Amer. Math. Soc.},
  volume  = {359},
  number  = {11},
  pages   = {5239--5256},
  year    = {2007},
  doi     = {10.1090/S0002-9947-07-03934-7}
}

@misc{GutweinLanglais2026,
  author        = {Gutwein, Dominik and Langlais, Thibault},
  title         = {Deformations of harmonic maps with conical singularities},
  year          = {2026},
  eprint        = {2609.20588},
  archivePrefix = {arXiv},
  primaryClass  = {math.DG},
  note          = {arXiv:2609.20588 [math.DG]},
  doi           = {10.48550/arXiv.2609.20588}
}

@misc{Surkov2026,
  author        = {Surkov, Egor},
  title         = {Harmonic maps from spheres with lowest possible index and their properties},
  year          = {2026},
  eprint        = {2609.35589},
  archivePrefix = {arXiv},
  primaryClass  = {math.DG},
  note          = {arXiv:2609.35589 [math.DG]},
  doi           = {10.48550/arXiv.2609.35589}
}

@misc{ChenWangZhang2026,
  author        = {Chen, Qun and Wang, Guofang and Zhang, Mingwei},
  title         = {Harmonic maps from {$\mathbb{S}^{2n+1}$} into {$\mathbb{CP}^n$} with least {Morse} index},
  year          = {2026},
  eprint        = {2609.37945},
  archivePrefix = {arXiv},
  primaryClass  = {math.DG},
  note          = {arXiv:2609.37945 [math.DG]},
  doi           = {10.48550/arXiv.2609.37945}
}

@article{Urakawa1987-so,
  author  = {Urakawa, Hajime},
  title   = {Stability of harmonic maps and eigenvalues of the {Laplacian}},
  journal = {Trans. Amer. Math. Soc.},
  volume  = {301},
  number  = {2},
  pages   = {557--589},
  year    = {1987},
  doi     = {10.1090/S0002-9947-1987-0882704-8}
}

@article{Riviere2023-co,
  author  = {Rivi{\`e}re, Tristan},
  title   = {Harmonic maps from {$S^3$} into {$S^2$} with low {Morse} index},
  journal = {J. Differential Geom.},
  volume  = {125},
  number  = {1},
  pages   = {173--185},
  year    = {2023},
  doi     = {10.4310/jdg/1695236594}
}

@article{ElSoufi93,
  author   = {El Soufi, Ahmad},
  title    = {Applications harmoniques, immersions minimales et transformations conformes de la sph{\`e}re},
  journal  = {Compositio Math.},
  volume   = {85},
  number   = {3},
  pages    = {281--298},
  year     = {1993},
  mrnumber = {1214448}
}

@book{Brocker2010-sv,
  author    = {Br{\"o}cker, Theodor and tom Dieck, Tammo},
  title     = {Representations of Compact {Lie} Groups},
  series    = {Graduate Texts in Mathematics},
  volume    = {98},
  publisher = {Springer-Verlag},
  address   = {New York},
  year      = {1985},
  doi       = {10.1007/978-3-662-12918-0}
}

@article{Ikeda78,
  author  = {Ikeda, Akira and Taniguchi, Yoshiharu},
  title   = {Spectra and eigenforms of the {Laplacian} on {$S^n$} and {$P^n(\mathbb{C})$}},
  journal = {Osaka J. Math.},
  volume  = {15},
  number  = {3},
  pages   = {515--546},
  year    = {1978},
  doi     = {10.18910/6956}
}

\end{document}